\documentclass[11pt,a4paper,oneside]{amsart}
\usepackage[utf8]{inputenc}
\usepackage{amsmath}
\usepackage{amssymb}
\usepackage{mathptmx}
\usepackage{extpfeil}
\usepackage{array}
\usepackage{enumitem}
\usepackage[tableposition=top]{caption}
\usepackage{arydshln}

\DeclareMathOperator{\Hom}{Hom}

\DeclareMathOperator{\Pic}{Pic}
\DeclareMathOperator{\Aut}{Aut}
\DeclareMathOperator{\modf}{mod}

\DeclareMathOperator{\Ob}{Ob}

\newcommand{\smod}{\underline{\modf}}

\newcommand{\A}{\mathbb A}
\newcommand{\D}{\mathbb D}
\newcommand{\Db}{\mathcal D^b}
\newcommand{\Dc}{\mathcal D}
\newcommand{\E}{\mathbb E}
\newcommand{\Mcl}{\mathcal M}
\newcommand{\Pc}{\mathcal P}
\newcommand{\T}{\mathcal T}
\newcommand{\Z}{\mathbb Z}

\newtheorem{definition}{Definition}[section]
\newtheorem{theorem}[definition]{Theorem}

\newtheorem{corollary}[definition]{Corollary}
\newtheorem{proposition}[definition]{Proposition}
\theoremstyle{definition}
\newtheorem{example}[definition]{Example}
\usepackage{tikz}
\usetikzlibrary{calc}
\usetikzlibrary{arrows,decorations.pathmorphing,decorations.pathreplacing, decorations.markings, matrix} 

\tikzset{vertex/.style={circle,fill=black,inner sep=1pt,outer sep=2pt},
         mvertex/.style={rectangle,draw=black,thick,inner sep=2pt,outer sep=2pt},
         tvertex/.style={inner sep=1pt,font=\scriptsize},
         unvertex/.style={circle,fill=white,draw=white,inner sep=1pt},
         bvertex/.style={circle,fill=white,draw=white,inner sep=1pt, font=\scriptsize},
         fill1/.style={fill=black!15,draw=black!15},
         fill2/.style={fill=black!30,draw=black!30},
         fill12/.style={fill=black!45,draw=black!45},
         leadsto/.style={-angle 90,decorate,decoration=snake,very thick},
         cut/.style={decorate,decoration=saw,very thick}}

\newcommand{\replacevertex}[3][fill=white,draw=white]
 {
  \node at #2 [#1,circle,inner sep=1pt] {};
  \node #2 at #2 #3;
 }

\begin{document}
\title[Realizing orbit categories as stable module categories]{Realizing orbit categories as stable module categories - a complete classification}
\author[B. Grimeland]{Benedikte Grimeland}
\address{Department of Teacher Education\\Norwegian University of Science and Technology\\7491 Trondheim}
\email{benedikte.grimeland@ntnu.no}
\author[K. M. Jacobsen]{Karin M.\ Jacobsen}
\address{Department of Mathematical Sciences\\Norwegian University of Science and Technology\\7491 Trondheim}
\email{karin.marie.jacobsen@math.ntnu.no}
\date{}
\begin{abstract}
We classify all triangulated orbit categories of path-algebras of Dynkin diagrams that are triangle equivalent to a stable module category of a representation-finite self-injective standard algebra. For each triangulated orbit category \(\T\) we give an explicit description of a representation-finite self-injective standard algebra with stable module category triangle equivalent to \(\T\). 
\end{abstract}
\maketitle



\section{Introduction}
Let \(k\) be an algebraically closed field. In this paper we will focus on two types of triangulated categories with finitely many isomorphism classes of indecomposable objects: triangulated orbit categories of bounded derived categories of path algebras of Dynkin quivers of type \(\A,\D\) and \(\E\), and stable module categories of representation-finite self-injective algebras of Dynkin tree type. The triangulated categories we consider are \(\Hom\)-finite.

It is well-known that the stable module category of a self-injective algebra is a triangulated category. Riedtmann showed in \cite{riedtmannTysk} that all connected stable components of the AR-quiver of a representation-finite algebra are of Dynkin tree type. In two subsequent papers by Riedtmann \cite{riedtmannA} and Bretschner, L{\"a}ser and Riedtmann \cite{RmVenner}, a complete classification of all representation-finite self-injective algebras of Dynkin type is given in terms of their quivers with relations. 
Continuing their work, Asashiba gives an invariant under derived equivalence for representation-finite self-injective algebras, based on the shape of the AR-quiver \cite{asashiba1}\cite{asashiba2}, called the type of the algebra. Standard algebras of one type are stably equivalent, as well as derived equivalent. He also determines which types contain standard algebras.

Triangulated orbit categories have been well studied, see e.g.\ \cite{bmrrt}, \cite{ccs} and \cite{kellerTriOrb}. The orbit category of a triangulated category is not necessarily triangulated itself. However Keller showed that the orbit category \(\mathcal D^b(H)/F\) is triangulated for \(H\) a hereditary algebra, \(\Db(H)\) the bounded derived category of \(\modf A\), and with certain restrictions on the functor \(F\) \cite{kellerTriOrb}. In the case where \(F=\tau^{-1}[m-1]\) for \(m\in\mathbb N\), the orbit category \(\mathcal D^b(H)/F\) is known as the \(m\)-cluster category \(\mathcal C_m(H)\). The Calabi-Yau dimension of \(\mathcal C_m(H) \) is \(m\).

Keller and Reiten proved in \cite{KelReitMorita} that an algebraic triangulated category with Calabi-Yau dimension \(m\) that contains an \((m-1)\)-cluster tilting object \(T\) with a hereditary endomorphism algebra \(H\) such that \(\Hom (T,\Sigma^{-i}T)=0\) for \(i=0,\ldots, m-2\) is triangle equivalent to the \(m\)-cluster category \(\mathcal C_m(H)\). 

More recently in \cite{dugas}, Dugas was able to determine the Calabi-Yau dimension to some of the stable module categories of representation-finite self-injective algebras.

The theorem of Keller and Reiten, combined with the Calabi-Yau dimensions calculated by Dugas, was used by Holm and J{\o}rgensen \cite{holmJ} to classify which stable module categories of self-injective algebras are triangle equivalent to an \(m\)-cluster category. 

Another approach has been to use Galois coverings to study triangle equivalences between triangulated categories with some finiteness condition: 

Xiao and Zhu show in \cite{XiaoZhu}  that if \(\T\) is a locally finite triangulated category, its Auslander-Reiten quiver is of the form \(\Z\Delta/G\), where \(\Delta\) is a Dynkin diagram and \(G\) is an automorphism group of \(\Z\Delta\). For most quivers \(\Z\Delta/G\) they also give triangulated categories  where \(\Z\Delta/G\) is the Auslander-Reiten quiver. In doing this, they show that the orbit categories of the form \(\mathcal D^b(k\A_n)/\tau^m\) are equivalent to stable categories. 

In \cite{amiot}, Amiot reproves Xiao and Zhu's statement for categories with finitely many equivalence classes of indecomposable objects. Amiot also shows \cite[Thm.\ 7.2]{amiot} that any finite, standard, connected, algebraic, triangulated category is triangle equivalent to \(\Db( k\Delta)/\Phi\) for some Dynkin diagram \(\Delta\) and autoequivalence \(\Phi\).
This result reduces the problem of finding triangle equivalences between triangulated categories to finding isomorphisms between translation quivers. 

Using the result from Amiot, we classify all triangulated orbit categories of path algebras of Dynkin diagrams that are triangle equivalent to the stable module category of a representation-finite standard self-injective algebra. 
The orbit categories we consider are all of standard type. We cannot have an equivalence between a category of standard type and one of non-standard type, so we only need to consider self-injective algebras of standard type.

In Sections 2-5, we give an overview of the theory required for our result. In particular, section 4 contains a corollary to Amiot's theorem \cite[Thm.\ 7.2]{amiot} that is the basis for our main result.

Sections 6-8 contain the calculations, examples and results for Dynkin type \(\A, \D, \E\). The results are summed up in Section 9, where the main theorem is stated:

\begin{theorem}
Let \(\Delta\) be a Dynkin diagram and let \(\Phi\) be an autoequivalence such that \(\Db(k\Delta)\!/\Phi\) is triangulated. Let \(\Lambda\) a self-injective algebra.
The orbit category \(\mathcal C=\Db(k\Delta)\!/\Phi\) is triangle equivalent to \(\smod \Lambda\) exactly in the cases described in table \ref{SummaryTableInt}.
\begin{table}[h!tb]
\centering
\begin{tabular}{|>{$}l<{$}>{$}l<{$}|l|l|}
\hline
\multicolumn{2}{|l|}{\(\mathcal{C}\)} 		& \(\Lambda\) & Sec.\!\\\hline
\Db(k\A_r)/\tau^w			&r\geq 1, w\geq 1 												& Nakayama alg.\ \(N_{w,r+1}\)\! & \ref{Nakayama}\\\hline
\Db(k\A_r)/\tau^w\phi	 	&\begin{aligned}r&=2l+1,\, l\geq 1 \\ w&=rv,\, r\geq 1\end{aligned} 	& M\"obius alg.\ \(M_{l,v}\) 	&\ref{Mobius} \\\hline
\Db(k\D_r)/\tau^w	 		&r\geq 4, w=s(2r-3), s\geq 1 				& \(D_{n,s,1}\) 		& \ref{Drot1} \\\hline
\Db(k\D_r)/\tau^w\phi	 	&r\geq 4, w=s(2r-3), s\geq 1				& \(D_{n,s,2}\)		& \ref{Drot2} \\\hline
\Db(k\D_4)/\tau^{5w}\rho 	&w\geq 1									& \(D_{4,s,3}\)		& \ref{Drot3} \\\hline
\Db(k\D_r)/\tau^w	 		&\begin{aligned}r&=3m, m\geq 2\\ w&=s(2r-3)/3, s\geq 1, 3\nmid s\end{aligned}  \!	& \(D_{3m,\frac{s}{3},1}\) & \ref{D3m}\\\hline
\Db(k\E_r)/\tau^w	 		&\begin{aligned}	
						r=6 \text{ and }& w=11s \\ 
						r=7 \text{ and }& w=17s  \hspace{5pt}s\geq 1 \\
						r=8 \text{ and }& w=29s \\ \end{aligned} 	 & \(E_{r, s, 1}\) & \ref{E678}\\\hline
\Db(k\E_6)/\tau^w\phi 		& w=11s, s\geq 1 & \(E_{6,s,2}\) 	& \ref{E6twist} \\\hline
\end{tabular}
\caption{The cases up to triangulated equivalence where \(\mathcal C=\Db(k\Delta)/\Phi\) is triangle equivalent to \(\smod \Lambda\). The definitions of the self-injective algebras \(\Lambda\)  are stated in the sections as listed.}
\label{SummaryTableInt}
\end{table}
\end{theorem}

\section{Translation quivers, mesh categories and automorphism groups}
Translation quivers can be seen as an abstraction of the properties of AR-quivers. They are central in Riedtmann's classification of all self-injective algebra of Dynkin type \(\A\), \(\D\) and \(\E\). They are also related to the derived category, as we will see in theorem \ref{meshDeriv}. 
Background on translation quivers can be found in \cite{happel} and \cite{ARS}, from which we recall the following central definitions and results. 
\begin{definition}
We define a quiver \(Q=(Q_0,Q_1,s,t)\) to consist of a set of vertices \(Q_0\), a set of arrows \(Q_1\), a source map \(s\) and a target/sink map \(t\). 
\begin{description}[font=\bfseries]
\item[\(x^-\) and \(x^+\)] For a  vertex \(x\in Q_0\) we denote by \(x^-\) the set of direct predecessors of \(x\) in \(Q\), and by \(x^+\) the set of direct successors of \(x\) in \(Q\).
\item[Locally finite quiver] A quiver \(Q\) is called locally finite if for each \(x\in Q_0\) the sets \(x^-\) and \(x^+\) are finite.
\item[Translation quiver] Let \(\theta\) be an injective map from a subset of \(Q_0\) to \(Q_0\). The pair \((Q,\theta)\) is called a translation quiver if the following is satisfied:
	\begin{enumerate}
	\item \(Q\) has no loops and no multiple arrows
	\item For \(x\in Q_0\) such that \(\theta(x)\) is defined, we have that \(x^-=\theta(x)^+\)
	\end{enumerate}
	The map \(\theta\) is called the translation of the translation quiver \((Q,\theta)\).
	
	For \(x\in Q_0\) such that \(\theta(x)\) is defined, we can define a map \(\sigma\) on the arrows going into \(x\), by setting \(\sigma(z\rightarrow x)=(\theta(x)\rightarrow z)\). This map is a bijection on the respective sets of arrows.
\item[Stable translation quiver] A translation quiver \((Q,\theta)\) is called stable if \(\theta:Q_0\rightarrow Q_0\) is a bijection.
\item[Morphism of translation quivers]  Given two translation quivers \((Q,\theta)\) and \((Q^{'}\!,\theta^{'}\!)\), we define a morphism \(f:(Q,\theta)\rightarrow (Q^{'},\theta^{'})\) as a pair of maps \(f_0:Q_0\rightarrow Q^{'}_0\) and \(f_1:Q_1\rightarrow Q^{'}_1\) such that 
	\begin{itemize}
	\item if \(\alpha\in Q_1\), and \(\alpha:x\rightarrow y\) then \(f_1(\alpha)\in Q^{'}_1\) is the arrow 			       	\(f_1(\alpha):f_0(x)\rightarrow f_0(y)\).
	\item for all vertices \(x\in Q\) where \(\theta\) is defined we have \(f_0(\theta(x))=\theta^{'}(f_0(x))\).
	\end{itemize}
\item[Isomorphism of translation quivers] A morphism of translation quivers \(f:(Q,\theta)\rightarrow (Q^{'},\theta^{'})\) is an isomorphism if it has an inverse. The inverse is a morphism of translation quivers \(g:(Q^{'},\theta^{'})\rightarrow (Q,\theta)\) such that \(g\circ f\) is the identity on \((Q,\theta)\) and \(f\circ g\) is the identity on \((Q^{'},\theta^{'})\).
\end{description}

\end{definition}

For a quiver \(\Delta=(\Delta_0, \Delta_1)\) without loops, we can define a stable translation quiver \((\Z\Delta, \theta)\) as follows: 
\begin{description}[font=\bfseries]
\item[Vertices] The elements of \(\Z\times \Delta_0\).
\item[Arrows] For any arrow \(i\rightarrow j\) in \(\Delta_1\) and any \(n\in \Z\), we have arrows \((n,i)\rightarrow (n,j)\) and \((n, j)\rightarrow (n+1,i)\).
\item[Translation] Given by \(\theta(n,i)=(n-1,i)\).
\end{description}

Our focus will be on translation quivers of the form \((\mathbb Z\Delta,\theta)\) for \(\Delta\) of Dynkin type \(\A,\D\) and \(\E\). We use the following orientation on the Dynkin diagrams:
\def\arraystretch{0}
\begin{center}
\begin{tabular}{m{1cm}m{10cm}}
\(\A_r\): &
\begin{tikzpicture}[anchor=west]
\node(1) at (0,0) {1};
\node(2) at (1,0) {2};
\node(dots) at (2,0) {\(\cdots\)};
\node(-) at (3.5,0) {\(r-1\)};
\node(r) at (5,0) {\(r\)};
\draw[->] (1) -- (2);
\draw[->] (2) -- (dots);
\draw[->] (dots) -- (-);
\draw[->] (-) -- (r);
\end{tikzpicture}
\\
\(\D_r\):&
\begin{tikzpicture}[anchor=west]
\node(1) at (0,0) {1};
\node(2) at (1,0) {2};
\node(dots) at (2,0) {\(\cdots\)};
\node(-2) at (3.5,0) {\(r-2\)};
\node(-) at (5,-.5) {\(r-1\)};
\node(r) at (5,.5) {\(r\)};
\draw[->] (1) -- (2);
\draw[->] (2) -- (dots);
\draw[->] (dots) -- (-2);
\draw[->] (-2) -- (r);
\draw[->] (-2) -- (-);
\end{tikzpicture}
\\
\(\E_r\):&
\begin{tikzpicture}
\node(1) at (0,0) {1};
\node(2) at (1,0) {2};
\node(3) at (2,0) {3};
\node(dots) at (3,0) {\(\cdots\)};
\node(-2) at (4.2,0) {\(r-2\)};
\node(-) at (5,0)[anchor=west] {\(r-1\)};
\node(r) at (2,1) {\(r\)};
\draw[->] (1) -- (2);
\draw[->] (2) -- (3);
\draw[->] (3) -- (r);
\draw[->] (3) -- (dots);
\draw[->] (dots) -- (-2);
\draw[->] (-2) -- (-);
\end{tikzpicture}
\\

\end{tabular}
\end{center}
\def\arraystretch{1}

The stable translation quivers \((\mathbb Z\A_r,\!\theta)\), \((\mathbb Z\D_r,\theta)\), \((\mathbb Z\E_6,\theta)\), \((\mathbb Z\E_7,\theta)\) and \((\mathbb Z\E_8,\theta)\) are shown in figure \ref{translationq}. 

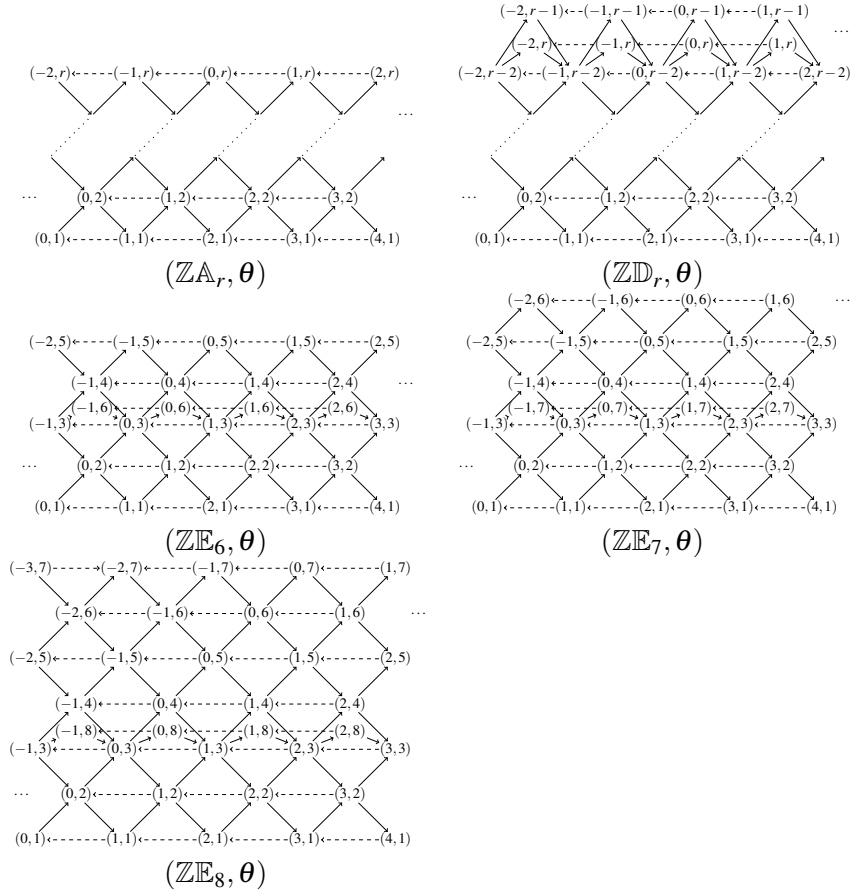
\begin{figure}[b]
\begin{tabular}{cc}
\scalebox{0.5}{
\begin{tikzpicture}[scale=1.1, inner sep=.5pt]
\node (01) at (0,0) {\((0,1)\)};
\node (11) at (2,0) {\((1,1)\)};
\node (21) at (4,0) {\((2,1)\)};
\node (31) at (6,0) {\((3,1)\)};
\node (41) at (8,0) {\((4,1)\)};
\node (02) at (1,1) {\((0,2)\)};
\node (12) at (3,1) {\((1,2)\)};
\node (22) at (5,1) {\((2,2)\)};
\node (32) at (7,1) {\((3,2)\)};

\node (-13) at (0,2) {};
\node (03) at (2,2) {};
\node (13) at (4,2) {};
\node (23) at (6,2) {};
\node (33) at (8,2) {};

\node (-1r-1) at (1,3) {};
\node (0r-1) at (3,3) {};
\node (1r-1) at (5,3) {};
\node (2r-1) at (7,3) {};

\node (-2r) at (0,4) {\((-2,r)\)};
\node (-1r) at (2,4) {\((-1,r)\)};
\node (0r) at (4,4) {\((0,r)\)};
\node (1r) at (6,4) {\((1,r)\)};
\node (2r) at (8,4) {\((2,r)\)};

\node (vdots) at (-.5,1) {\(\cdots\)};
\node (hdots) at (8.5,3) {\(\cdots\)};

\draw[->] (01)--(02);
\draw[->] (11)--(12);
\draw[->] (21)--(22);
\draw[->] (31)--(32);

\draw[->] (02)--(03);
\draw[->] (12)--(13);
\draw[->] (22)--(23);
\draw[->] (32)--(33);

\draw[->] (-1r-1)--(-1r);
\draw[->] (0r-1)--(0r);
\draw[->] (1r-1)--(1r);
\draw[->] (2r-1)--(2r);

\draw[->] (02)--(11);
\draw[->] (12)--(21);
\draw[->] (22)--(31);
\draw[->] (32)--(41);

\draw[->] (-13)--(02);
\draw[->] (03)--(12);
\draw[->] (13)--(22);
\draw[->] (23)--(32);

\draw[->] (-2r)--(-1r-1);
\draw[->] (-1r)--(0r-1);
\draw[->] (0r)--(1r-1);
\draw[->] (1r)--(2r-1);
\draw[thick, loosely dotted] (-13)--(-1r-1);
\draw[thick, loosely dotted] (03)--(0r-1);
\draw[thick, loosely dotted] (13)--(1r-1);
\draw[thick, loosely dotted] (23)--(2r-1);

\draw[dashed, ->] (11) -- (01);
\draw[dashed, ->] (21) -- (11);
\draw[dashed, ->] (31) -- (21);
\draw[dashed, ->] (41) -- (31);

\draw[dashed, ->] (12) -- (02);
\draw[dashed, ->] (22) -- (12);
\draw[dashed, ->] (32) -- (22);

\draw[dashed, ->] (-1r) -- (-2r);
\draw[dashed, ->] (0r) -- (-1r);
\draw[dashed, ->] (1r) -- (0r);
\draw[dashed, ->] (2r) -- (1r);

\end{tikzpicture}}
&
\scalebox{.5}{
\begin{tikzpicture}[scale=1.1, inner sep=.5pt]
\node (01) at (0,0) {\((0,1)\)};
\node (11) at (2,0) {\((1,1)\)};
\node (21) at (4,0) {\((2,1)\)};
\node (31) at (6,0) {\((3,1)\)};
\node (41) at (8,0) {\((4,1)\)};

\node (02) at (1,1) {\((0,2)\)};
\node (12) at (3,1) {\((1,2)\)};
\node (22) at (5,1) {\((2,2)\)};
\node (32) at (7,1) {\((3,2)\)};

\node (-13) at (0,2) {};
\node (03) at (2,2) {};
\node (13) at (4,2) {};
\node (23) at (6,2) {};
\node (33) at (8,2) {};

\node (-1r-1) at (1,3) {};
\node (0r-1) at (3,3) {};
\node (1r-1) at (5,3) {};
\node (2r-1) at (7,3) {};

\node (-2r) at (0,4) {\((-2,r-2)\)};
\node (-1r) at (2,4) {\((-1,r-2)\)};
\node (0r) at (4,4) {\((0,r-2)\)};
\node (1r) at (6,4) {\((1,r-2)\)};
\node (2r) at (8,4) {\((2,r-2)\)};

\node (-2l) at (1,4.7) {\((-2, r)\)};
\node (-2h) at (1,5.5) {\((-2, r-1)\)};
\node (-1l) at (3,4.7) {\((-1, r)\)};
\node (-1h) at (3,5.5) {\((-1, r-1)\)};
\node (0l)  at (5,4.7) {\((0, r)\)};
\node (0h)  at (5,5.5) {\((0, r-1)\)};
\node (1l)  at (7,4.7) {\((1, r)\)};
\node (1h)  at (7,5.5) {\((1, r-1)\)};

\node (vdots) at (-.5,1) {\(\cdots\)};
\node (hdots) at (8.4,5) {\(\cdots\)};

\draw[->] (01)--(02);
\draw[->] (11)--(12);
\draw[->] (21)--(22);
\draw[->] (31)--(32);

\draw[->] (02)--(03);
\draw[->] (12)--(13);
\draw[->] (22)--(23);
\draw[->] (32)--(33);

\draw[->] (-1r-1)--(-1r);
\draw[->] (0r-1)--(0r);
\draw[->] (1r-1)--(1r);
\draw[->] (2r-1)--(2r);

\draw[->] (02)--(11);
\draw[->] (12)--(21);
\draw[->] (22)--(31);
\draw[->] (32)--(41);

\draw[->] (-13)--(02);
\draw[->] (03)--(12);
\draw[->] (13)--(22);
\draw[->] (23)--(32);

\draw[->] (-2r)--(-1r-1);
\draw[->] (-1r)--(0r-1);
\draw[->] (0r)--(1r-1);
\draw[->] (1r)--(2r-1);

\draw[->] (-2r)--(-2l);
\draw[->] (-2r)--(-2h);
\draw[->] (-2l)--(-1r);
\draw[->] (-2h)--(-1r);
\draw[->] (-1r)--(-1l);
\draw[->] (-1r)--(-1h);
\draw[->] (-1l)--(0r);
\draw[->] (-1h)--(0r);
\draw[->] (0r)--(0l);
\draw[->] (0r)--(0h);
\draw[->] (0l)--(1r);
\draw[->] (0h)--(1r);
\draw[->] (1r)--(1l);
\draw[->] (1r)--(1h);
\draw[->] (1l)--(2r);
\draw[->] (1h)--(2r);

\draw[thick, loosely dotted] (-13)--(-1r-1);
\draw[thick, loosely dotted] (03)--(0r-1);
\draw[thick, loosely dotted] (13)--(1r-1);
\draw[thick, loosely dotted] (23)--(2r-1);

\draw[dashed, ->] (11) -- (01);
\draw[dashed, ->] (21) -- (11);
\draw[dashed, ->] (31) -- (21);
\draw[dashed, ->] (41) -- (31);

\draw[dashed, ->] (12) -- (02);
\draw[dashed, ->] (22) -- (12);
\draw[dashed, ->] (32) -- (22);

\draw[dashed, ->] (-1r) -- (-2r);
\draw[dashed, ->] (0r) -- (-1r);
\draw[dashed, ->] (1r) -- (0r);
\draw[dashed, ->] (2r) -- (1r);

\draw[dashed, ->] (-1l) -- (-2l);
\draw[dashed, ->] (0l) -- (-1l);
\draw[dashed, ->] (1l) -- (0l);

\draw[dashed, ->] (-1h) -- (-2h);
\draw[dashed, ->] (0h) -- (-1h);
\draw[dashed, ->] (1h) -- (0h);

\end{tikzpicture}}
\\
\((\mathbb Z \A_r,\theta)\)&
\((\mathbb Z \D_r,\theta)\)\\

\scalebox{.5}{
\begin{tikzpicture}[scale=1.1, inner sep=.5pt]
\node (01) at (0,0) {\((0,1)\)};
\node (11) at (2,0) {\((1,1)\)};
\node (21) at (4,0) {\((2,1)\)};
\node (31) at (6,0) {\((3,1)\)};
\node (41) at (8,0) {\((4,1)\)};

\node (02) at (1,1) {\((0,2)\)};
\node (12) at (3,1) {\((1,2)\)};
\node (22) at (5,1) {\((2,2)\)};
\node (32) at (7,1) {\((3,2)\)};

\node (-13) at (0,2) {\((-1,3)\)};
\node (03) at (2,2) {\((0,3)\)};
\node (13) at (4,2) {\((1,3)\)};
\node (23) at (6,2) {\((2,3)\)};
\node (33) at (8,2) {\((3,3)\)};

\node (-14) at (1,3) {\((-1,4)\)};
\node (04) at (3,3) {\((0,4)\)};
\node (14) at (5,3) {\((1,4)\)};
\node (24) at (7,3) {\((2,4)\)};

\node (-25) at (0,4) {\((-2,5)\)};
\node (-15) at (2,4) {\((-1,5)\)};
\node (05) at (4,4) {\((0,5)\)};
\node (15) at (6,4) {\((1,5)\)};
\node (25) at (8,4) {\((2,5)\)};

\node (-16) at (1,2.4) {\((-1,6)\)};
\node (06) at (3,2.4) {\((0,6)\)};
\node (16) at (5,2.4) {\((1,6)\)};
\node (26) at (7,2.4) {\((2,6)\)};

\node (vdots) at (-.5,1) {\(\cdots\)};
\node (hdots) at (8.5,3) {\(\cdots\)};

\draw[->] (01)--(02);
\draw[->] (11)--(12);
\draw[->] (21)--(22);
\draw[->] (31)--(32);

\draw[->] (02)--(11);
\draw[->] (12)--(21);
\draw[->] (22)--(31);
\draw[->] (32)--(41);

\draw[->] (02)--(03);
\draw[->] (12)--(13);
\draw[->] (22)--(23);
\draw[->] (32)--(33);

\draw[->] (-13)--(02);
\draw[->] (03)--(12);
\draw[->] (13)--(22);
\draw[->] (23)--(32);

\draw[->] (-13)--(-14);
\draw[->] (03)--(04);
\draw[->] (13)--(14);
\draw[->] (23)--(24);

\draw[->] (-13)--(-16);
\draw[->] (03)--(06);
\draw[->] (13)--(16);
\draw[->] (23)--(26);

\draw[->] (-14)--(03);
\draw[->] (04)--(13);
\draw[->] (14)--(23);
\draw[->] (24)--(33);

\draw[->] (-16)--(03);
\draw[->] (06)--(13);
\draw[->] (16)--(23);
\draw[->] (26)--(33);

\draw[->] (-14)--(-15);
\draw[->] (04)--(05);
\draw[->] (14)--(15);
\draw[->] (24)--(25);

\draw[->] (-25)--(-14);
\draw[->] (-15)--(04);
\draw[->] (05)--(14);
\draw[->] (15)--(24);

\draw[dashed, ->] (11) -- (01);
\draw[dashed, ->] (21) -- (11);
\draw[dashed, ->] (31) -- (21);
\draw[dashed, ->] (41) -- (31);

\draw[dashed, ->] (12) -- (02);
\draw[dashed, ->] (22) -- (12);
\draw[dashed, ->] (32) -- (22);

\draw[dashed, ->] (03) -- (-13);
\draw[dashed, ->] (13) -- (03);
\draw[dashed, ->] (23) -- (13);
\draw[dashed, ->] (33) -- (23);

\draw[dashed, ->] (04) -- (-14);
\draw[dashed, ->] (14) -- (04);
\draw[dashed, ->] (24) -- (14);

\draw[dashed, ->] (06) -- (-16);
\draw[dashed, ->] (16) -- (06);
\draw[dashed, ->] (26) -- (16);

\draw[dashed, ->] (-15) -- (-25);
\draw[dashed, ->] (05) -- (-15);
\draw[dashed, ->] (15) -- (05);
\draw[dashed, ->] (25) -- (15);
\end{tikzpicture}}
%
&
\scalebox{.5}{
\begin{tikzpicture}[scale=1.1, inner sep=.5pt]
\node (01) at (0,0) {\((0,1)\)};
\node (11) at (2,0) {\((1,1)\)};
\node (21) at (4,0) {\((2,1)\)};
\node (31) at (6,0) {\((3,1)\)};
\node (41) at (8,0) {\((4,1)\)};

\node (02) at (1,1) {\((0,2)\)};
\node (12) at (3,1) {\((1,2)\)};
\node (22) at (5,1) {\((2,2)\)};
\node (32) at (7,1) {\((3,2)\)};

\node (-13) at (0,2) {\((-1,3)\)};
\node (03) at (2,2) {\((0,3)\)};
\node (13) at (4,2) {\((1,3)\)};
\node (23) at (6,2) {\((2,3)\)};
\node (33) at (8,2) {\((3,3)\)};

\node (-14) at (1,3) {\((-1,4)\)};
\node (04) at (3,3) {\((0,4)\)};
\node (14) at (5,3) {\((1,4)\)};
\node (24) at (7,3) {\((2,4)\)};

\node (-25) at (0,4) {\((-2,5)\)};
\node (-15) at (2,4) {\((-1,5)\)};
\node (05) at (4,4) {\((0,5)\)};
\node (15) at (6,4) {\((1,5)\)};
\node (25) at (8,4) {\((2,5)\)};

\node (-26) at (1,5) {\((-2,6)\)};
\node (-16) at (3,5) {\((-1,6)\)};
\node (06) at (5,5) {\((0,6)\)};
\node (16) at (7,5) {\((1,6)\)};

\node (-17) at (1,2.4) {\((-1,7)\)};
\node (07) at (3,2.4) {\((0,7)\)};
\node (17) at (5,2.4) {\((1,7)\)};
\node (27) at (7,2.4) {\((2,7)\)};

\node (vdots) at (-.5,1) {\(\cdots\)};
\node (hdots) at (8.5,5) {\(\cdots\)};

\draw[->] (01)--(02);
\draw[->] (11)--(12);
\draw[->] (21)--(22);
\draw[->] (31)--(32);

\draw[->] (02)--(11);
\draw[->] (12)--(21);
\draw[->] (22)--(31);
\draw[->] (32)--(41);

\draw[->] (02)--(03);
\draw[->] (12)--(13);
\draw[->] (22)--(23);
\draw[->] (32)--(33);

\draw[->] (-13)--(02);
\draw[->] (03)--(12);
\draw[->] (13)--(22);
\draw[->] (23)--(32);

\draw[->] (-13)--(-14);
\draw[->] (03)--(04);
\draw[->] (13)--(14);
\draw[->] (23)--(24);

\draw[->] (-13)--(-17);
\draw[->] (03)--(07);
\draw[->] (13)--(17);
\draw[->] (23)--(27);

\draw[->] (-14)--(03);
\draw[->] (04)--(13);
\draw[->] (14)--(23);
\draw[->] (24)--(33);

\draw[->] (-17)--(03);
\draw[->] (07)--(13);
\draw[->] (17)--(23);
\draw[->] (27)--(33);

\draw[->] (-14)--(-15);
\draw[->] (04)--(05);
\draw[->] (14)--(15);
\draw[->] (24)--(25);

\draw[->] (-25)--(-14);
\draw[->] (-15)--(04);
\draw[->] (05)--(14);
\draw[->] (15)--(24);

\draw[->] (-25)--(-26);
\draw[->] (-15)--(-16);
\draw[->] (05)--(06);
\draw[->] (15)--(16);

\draw[->] (-26)--(-15);
\draw[->] (-16)--(05);
\draw[->] (06)--(15);
\draw[->] (16)--(25);
\draw[dashed, ->] (11) -- (01);
\draw[dashed, ->] (21) -- (11);
\draw[dashed, ->] (31) -- (21);
\draw[dashed, ->] (41) -- (31);

\draw[dashed, ->] (12) -- (02);
\draw[dashed, ->] (22) -- (12);
\draw[dashed, ->] (32) -- (22);

\draw[dashed, ->] (03) -- (-13);
\draw[dashed, ->] (13) -- (03);
\draw[dashed, ->] (23) -- (13);
\draw[dashed, ->] (33) -- (23);

\draw[dashed, ->] (04) -- (-14);
\draw[dashed, ->] (14) -- (04);
\draw[dashed, ->] (24) -- (14);

\draw[dashed, ->] (07) -- (-17);
\draw[dashed, ->] (17) -- (07);
\draw[dashed, ->] (27) -- (17);

\draw[dashed, ->] (-15) -- (-25);
\draw[dashed, ->] (05) -- (-15);
\draw[dashed, ->] (15) -- (05);
\draw[dashed, ->] (25) -- (15);

\draw[dashed, ->] (-16) -- (-26);
\draw[dashed, ->] (06) -- (-16);
\draw[dashed, ->] (16) -- (06);
\end{tikzpicture}}
%
\\
\((\mathbb Z \E_6,\theta)\)&
\((\mathbb Z \E_7,\theta)\)\\
\scalebox{.5}{
\begin{tikzpicture}[scale=1.2, inner sep=.5pt]
\node (01) at (0,0) {\((0,1)\)};
\node (11) at (2,0) {\((1,1)\)};
\node (21) at (4,0) {\((2,1)\)};
\node (31) at (6,0) {\((3,1)\)};
\node (41) at (8,0) {\((4,1)\)};

\node (02) at (1,1) {\((0,2)\)};
\node (12) at (3,1) {\((1,2)\)};
\node (22) at (5,1) {\((2,2)\)};
\node (32) at (7,1) {\((3,2)\)};

\node (-13) at (0,2) {\((-1,3)\)};
\node (03) at (2,2) {\((0,3)\)};
\node (13) at (4,2) {\((1,3)\)};
\node (23) at (6,2) {\((2,3)\)};
\node (33) at (8,2) {\((3,3)\)};

\node (-14) at (1,3) {\((-1,4)\)};
\node (04) at (3,3) {\((0,4)\)};
\node (14) at (5,3) {\((1,4)\)};
\node (24) at (7,3) {\((2,4)\)};

\node (-25) at (0,4) {\((-2,5)\)};
\node (-15) at (2,4) {\((-1,5)\)};
\node (05) at (4,4) {\((0,5)\)};
\node (15) at (6,4) {\((1,5)\)};
\node (25) at (8,4) {\((2,5)\)};

\node (-26) at (1,5) {\((-2,6)\)};
\node (-16) at (3,5) {\((-1,6)\)};
\node (06) at (5,5) {\((0,6)\)};
\node (16) at (7,5) {\((1,6)\)};

\node (-37) at (0,6) {\((-3,7)\)};
\node (-27) at (2,6) {\((-2,7)\)};
\node (-17) at (4,6) {\((-1,7)\)};
\node (07) at (6,6) {\((0,7)\)};
\node (17) at (8,6) {\((1,7)\)};

\node (-18) at (1,2.4) {\((-1,8)\)};
\node (08) at (3,2.4) {\((0,8)\)};
\node (18) at (5,2.4) {\((1,8)\)};
\node (28) at (7,2.4) {\((2,8)\)};

\node (vdots) at (-0.2,1) {\(\cdots\)};
\node (hdots) at (8.5,5) {\(\cdots\)};

\draw[->] (01)--(02);
\draw[->] (11)--(12);
\draw[->] (21)--(22);
\draw[->] (31)--(32);

\draw[->] (02)--(11);
\draw[->] (12)--(21);
\draw[->] (22)--(31);
\draw[->] (32)--(41);

\draw[->] (02)--(03);
\draw[->] (12)--(13);
\draw[->] (22)--(23);
\draw[->] (32)--(33);

\draw[->] (-13)--(02);
\draw[->] (03)--(12);
\draw[->] (13)--(22);
\draw[->] (23)--(32);

\draw[->] (-13)--(-14);
\draw[->] (03)--(04);
\draw[->] (13)--(14);
\draw[->] (23)--(24);

\draw[->] (-13)--(-18);
\draw[->] (03)--(08);
\draw[->] (13)--(18);
\draw[->] (23)--(28);

\draw[->] (-14)--(03);
\draw[->] (04)--(13);
\draw[->] (14)--(23);
\draw[->] (24)--(33);

\draw[->] (-18)--(03);
\draw[->] (08)--(13);
\draw[->] (18)--(23);
\draw[->] (28)--(33);

\draw[->] (-14)--(-15);
\draw[->] (04)--(05);
\draw[->] (14)--(15);
\draw[->] (24)--(25);

\draw[->] (-25)--(-14);
\draw[->] (-15)--(04);
\draw[->] (05)--(14);
\draw[->] (15)--(24);

\draw[->] (-25)--(-26);
\draw[->] (-15)--(-16);
\draw[->] (05)--(06);
\draw[->] (15)--(16);

\draw[->] (-26)--(-15);
\draw[->] (-16)--(05);
\draw[->] (06)--(15);
\draw[->] (16)--(25);

\draw[->] (-26)--(-27);
\draw[->] (-16)--(-17);
\draw[->] (06)--(07);
\draw[->] (16)--(17);

\draw[->] (-37)--(-26);
\draw[->] (-27)--(-16);
\draw[->] (-17)--(06);
\draw[->] (07)--(16);

\draw[dashed, ->] (11) -- (01);
\draw[dashed, ->] (21) -- (11);
\draw[dashed, ->] (31) -- (21);
\draw[dashed, ->] (41) -- (31);

\draw[dashed, ->] (12) -- (02);
\draw[dashed, ->] (22) -- (12);
\draw[dashed, ->] (32) -- (22);

\draw[dashed, ->] (03) -- (-13);
\draw[dashed, ->] (13) -- (03);
\draw[dashed, ->] (23) -- (13);
\draw[dashed, ->] (33) -- (23);

\draw[dashed, ->] (04) -- (-14);
\draw[dashed, ->] (14) -- (04);
\draw[dashed, ->] (24) -- (14);

\draw[dashed, ->] (08) -- (-18);
\draw[dashed, ->] (18) -- (08);
\draw[dashed, ->] (28) -- (18);

\draw[dashed, ->] (-15) -- (-25);
\draw[dashed, ->] (05) -- (-15);
\draw[dashed, ->] (15) -- (05);
\draw[dashed, ->] (25) -- (15);

\draw[dashed, ->] (-16) -- (-26);
\draw[dashed, ->] (06) -- (-16);
\draw[dashed, ->] (16) -- (06);

\draw[dashed, ->] (-37) -- (-27);
\draw[dashed, ->] (-17) -- (-27);
\draw[dashed, ->] (07) -- (-17);
\draw[dashed, ->] (17) -- (07);
\end{tikzpicture}}
\\
\((\mathbb Z \E_8,\theta)\)&
\end{tabular}
\caption{Translation quivers of Dynkin diagrams}
\label{translationq}
\end{figure}
The set of automorphisms on a translation quiver \((Q, \theta)\) forms a group \(A\).  A \emph{group of automorphisms} of \((Q,\theta)\) is a subgroup of \(A\). 

\begin{definition}
Let \(G\) be a group of automorphisms of a translation quiver \((Q,\theta)\). The group \(G\) is called admissible if each orbit of \(G\) intersects the set \(\left\{x\right\}\cup x^+\) in at most one point, and intersects the set \(\left\{x\right\}\cup x^-\) in at most one point for each \(x\in Q_0\). 
\end{definition}

Given a (stable) translation quiver \((Q,\theta)\) and an admissible group \(G\) of  automorphisms of \((Q,\theta)\), one can form the (stable) translation quiver \((Q,\theta)/G\), where \((Q/G)_0=Q_0/G\) and \((Q/G)_1=Q_1/G\). The maps \(s,t\) and \(\theta\) are induced by the corresponding maps of \((Q, \theta)\) \cite{riedtmannTysk}. For the stable translation quivers given by \(\mathbb Z\Delta\), where \(\Delta\) is a Dynkin diagram, all admissible automorphism groups are known \cite{riedtmannTysk}\cite{amiot}.

For a translation quiver \((Q,\theta)\), we can define a path category \(\Pc(Q,\theta)\) as follows:
\begin{description}[font=\bfseries]
\item[Objects] The objects are the vertices of \(Q\).
\item[Morphisms] For a pair of objects \(x,y\in\Ob\Pc(Q,\theta)\) we let \(\Hom_\Pc(x,y)\) be the \(k\)-vector space with basis given by the set of all paths from \(x\) to \(y\).
\end{description}
There is a mesh ideal in \(\Pc(Q,\theta)\) generated by relations \[m_x=\sum_{\alpha: z\rightarrow x}\alpha \sigma(\alpha)\] for any \(x\) where \(\theta(x)\) is defined.
We get the mesh category \(\Mcl(Q,\theta)\) by taking the quotient of the path category by the mesh ideal.

\begin{theorem}[\cite{happel}]\label{meshDeriv}
Let \(\Delta\) be a Dynkin quiver.
The category of indecomposables in \(\Db(k\Delta)\) is equivalent to \(\Mcl_{(\Z\Delta, \theta)}\) as an additive category. 
\end{theorem}

The category \(\Db(k\Delta)\) has Auslander-Reiten triangles \cite{happel} and the Auslander-Reiten translation \(\tau\) commutes with suspension functor. 
Thus an autoequivalence on \(\Db(k\Delta)\) induces an automorphism on \((\Z\Delta, \theta)\).
In particular, the suspension functor on \(\Db(k\Delta)\) induces an automorphism on \(\Mcl_{(\Z\Delta, \theta)}\), of which we give the details in Table \ref{tableOfS}.
\setlength\dashlinedash{0.75pt}
\setlength\dashlinegap{1.5pt}
\setlength\arrayrulewidth{0.3pt}

\begin{table}[b]
\begin{center}
\begin{tabular}{m{3.1cm}|m{7.6cm}}
Translation quiver & Automorphism \(S\)\\\hline
\((\mathbb Z\A_n,\theta)\) 	& \(S(p,q)=(p+q, n+1-q)\)		\\ 
[1ex]
\hdashline

\((\mathbb Z\D_n,\theta)\) \(n\) even 	& \(S=\theta^{-(n+1)}\) 		\\ 
[1ex]
\hdashline
\((\mathbb Z\D_n,\theta)\) \(n\) odd 	& \(S=\theta^{-(n+1)}\xi\), where \(\xi\) is the automorphism on \((\mathbb Z\D_n,\theta)\) which exchanges the vertices \((x,r)\) and \((x,r-1)\) for \(x\in\mathbb Z\)		\\ 
[1ex]
\hdashline
\((\mathbb Z\E_6,\theta)\) 	& \(S=\xi\theta^{-6}\), where \(\xi\) is the automorphism on \((\mathbb Z\E_6,\theta)\) exchanging \((x,5)\) with \((x+2,1)\) and \((y,4)\) with \((y+1,2)\) for \(x,y\in\mathbb Z\)		\\
[1ex]
\hdashline
\((\mathbb Z\E_7,\theta)\) 	&	 \(S=\theta^{-9}\)			\\
[1ex]\hdashline
\((\mathbb Z\E_8,\theta)\) 	& \(S=\theta^{-15}\)			\\
\end{tabular}
\caption{The definition of the automorphism S in translation quivers of Dynkin type}\label{tableOfS}
\end{center}
\end{table}

Conversely, an automorphism on \((\Z\Delta, \theta\!)\) induces an autoequivalence on \(\Db\!(k\Delta)\):
\begin{theorem}[{\cite[Thm. 3.8]{MiyachiYekutieli}}]\label{Picard}
Let \(\Delta\) be a Dynkin diagram.
Let \(\Dc\Pic_k(k\Delta)\) be the derived Picard group,  i.\ e.\ the group of automorphisms on \(\Db(k\Delta)\) induced by two-sided tilting complexes. Let \(\Aut(\Z\Delta,\theta)\) be the group of automorphisms on \((\Z\Delta, \theta)\). Then \[\Dc\Pic_k(k\Delta)\cong \Aut(\Z\Delta,\theta).\]
\end{theorem}
This result will be used in Section \ref{typeD}.

\section{Orbit Categories}
Throughout the rest of this paper we will assume \(k\) to be an algebraically closed field. 

\begin{definition}
Given an additive category  \(\mathcal A\) and an automorphism \(F:\mathcal A\rightarrow \mathcal A \), we define the quotient functor \(\pi: \mathcal A\rightarrow \mathcal A/F\), where \(\mathcal A/F\) is the orbit category. The orbit category has the same objects as \(\mathcal A\), and morphisms given by \(\Hom_{\mathcal A/F}(X,Y)=\bigoplus_{n\in\mathbb Z}\Hom_\mathcal A(X,F^nY)\).
\end{definition}
We can replace the automorphism by an autoequivalence, see e.\ g.\ \cite{AsashibaGabriel}.
Certain orbit categories of triangulated categories were shown by Keller in \cite{kellerTriOrb} to be triangulated:
\begin{theorem}[\cite{kellerTriOrb}]\label{KellerTriOrb}
Let \(\mathcal H\) be an hereditary abelian k-category such that there is a triangle equivalence \[\mathcal D^b(A)\cong\mathcal D^b(\mathcal H),\] where \(A\) is a finite dimensional \(k\)-algebra. If \(F\) is a standard autoequivalence on \(\mathcal D^b(\mathcal H)\) such that 
\begin{itemize}
\item for each indecomposable object \(U\) of \(\mathcal H\) there are only finitely many objects \(F^iU\) that lie in \(\mathcal H\) for \(i\in\mathbb Z\).
\item there exist some integer \(N\geq0\) such that the \(F\)-orbit of each indecomposable object of \(\mathcal D^b(\mathcal H)\) contains an object \(U\left[n\right]\) for some \(0\leq n\leq N\) and some indecomposable object \(U\) of \(\mathcal H\).
\end{itemize}
Then the orbit category \(\mathcal O_F(\mathcal H):=\mathcal D^b(\mathcal H)/F\) is naturally a triangulated category, and the projection functor \(\pi:\mathcal D^b(\mathcal H)\rightarrow \mathcal O_F(\mathcal H)\) is a triangle functor. 
\end{theorem}
We now let \(\Delta\) be a Dynkin diagram, and consider the category \(\mathcal D^b(k\Delta)\). The AR-translation \(\tau\) and the suspension functor \(\left[1\right]\) satisfies the requirements on \(F\). In many cases, as we will see, so will the  composition \(\tau^m\left[n\right]\).

The AR-quiver of \(\mathcal D^b(k\Delta)\) is isomorphic as a translation quiver to \((\mathbb Z \Delta,\theta)\).  The action of \(\tau\) and \(\left[1\right]\) on the AR-quiver of \(\mathcal D^b(k \Delta)\) is equivalent to the action of respectively \(\theta\) and \(S\) on \((\mathbb Z \Delta,\theta)\). Hence, if \(\tau^m\left[n\right]\) satisfies the requirements on \(F\),  the AR-quiver of \(\mathcal D^b(k\Delta)/\tau^m\left[n\right]\) is isomorphic as a translation quiver to \((\mathbb Z \Delta,\theta)/(\theta^m S^n)\).

In \(\mathcal D^b(k\Delta)\) we know that \(\left[2\right]\cong\tau^{-h}\) where \(h\) is the Coxeter number of \(\Delta\) see \cite{GabrielAR}\cite{clusterKeller}. The Coxeter number is known to be \(n+1\) for \(\A_n\), \(2n-2\) for \(\D_n\), \(12\) for \(\E_6\), \(18\) for \(\E_7\) and \(30\) for \(\E_8\).

\section{Amiot's theorem}
A very important tool we will use is theorem \cite[Thm.\ 7.2]{amiot} by Amiot. We first need to give a definition of two special classes of triangulated categories.
\begin{definition}
A \(\Hom\)-finite triangulated category \(\mathcal T\) is called
\begin{description}[font=\bfseries]
\item[algebraic] if it is triangle equivalent to the stable category of a Frobenius category.
\item[standard] if the category of indecomposable objects of \(\T\) equivalent as a \(k\)-linear category to the mesh category \(\Mcl_{(k\Gamma, \tau)}\), where \(\Gamma\) is the AR-quiver of \(\mathcal T\).
\item[finite] if it has finitely many indecomposable objects up to isomorphism
\end{description}
\end{definition}

\begin{theorem}[{\cite[7.2]{amiot}}] \label{amiotMain}
Let \(\mathcal T\) be an indecomposable finite triangulated category which is algebraic and standard. Then there exists a Dynkin diagram \(\Delta\) of type \(\A\), \(\D\) or \(\E\), and an auto-equivalence \(\Phi\) on \(\mathcal D^b( k\Delta)\) such that \(\mathcal T\) is triangle equivalent to the orbit category \(\mathcal D^b( k\Delta)/\Phi\).
\end{theorem}
We specialize the theorem to deal with the cases we will use:

\begin{corollary}\label{sameARquiver}
Let \(\Lambda\) be a representation-finite, self-injective, basic algebra such that \(\smod \Lambda\) is of standard type.
Let \(\Delta\) be a Dynkin diagram, and let \(\Phi: \Db ( k\Delta)\rightarrow \Db ( k\Delta)\) be a functor such that \(\Db ( k\Delta)/\Phi\) is triangulated.

If the AR-quivers of \(\modf \Lambda\) and \(\Db ( k\Delta)/\Phi\) are isomorphic as translation quivers, then  \(\smod \Lambda\) and \(\Db ( k\Delta)/\Phi\) are equivalent as triangulated categories.
\end{corollary}
\begin{proof}
Obviously, \(\smod \Lambda\) is a finite standard triangulated category. It is algebraic, because \(\Lambda\) is self-injective and basic, and hence Frobenius.
By the proof of theorem \ref{amiotMain} in \cite{amiot}, the equivalence follows.
\end{proof}

\section{Self-injective Representation-finite Algebras}

Our aim is to use Claire Amiot's theorem to show that many orbit categories of bounded derived categories of hereditary algebras are actually realizable as stable module categories of self-injective algebras. In order to apply the theorem on the stable module categories of self-injective algebras, we need to know that the categories are algebraic and standard. It is clear that they are algebraic, as any representation-finite self-injective algebra is Frobenius. 

Asashiba has in his paper \cite{asashiba2} defined an invariant under derived and stable equivalence,  called the \emph{type} of the indecomposable representation-finite self-injective algebra. He shows that any two standard (resp. non-standard) representation-finite self-injective algebras have the same type if and only if they are derived equivalent, and also if and only if they are stably equivalent. In the appendix to \cite{asashiba1} a list of algebras, in terms of quivers with relations, is given for each type defined in \cite{asashiba2}. In sections \ref{typeA}, \ref{typeD} and \ref{typeE}, we make use of the explicit representatives for each type, and give the details of equivalent orbit categories and stable module categories of self-injective indecomposable algebras.

We give a brief summary of the classification of Asashiba. 

\begin{definition}[\cite{asashiba2}]\label{defType}
Let \(\Delta\) be a Dynkin diagram type \(\A,\D,\E_6,\E_7\) or \(\E_8\).
We define the type of a representation-finite self-injective indecomposable algebra \(\Lambda\) to be a triple \((\Delta(\!\Lambda\!), f(\!\Lambda\!), t(\!\Lambda\!))\). The parameters are defined as follows:
\begin{description}[font=\bfseries]
\item[\(\Delta(\Lambda)\)] the tree type of \(\Lambda\) (for this definition, we write \(\Delta=\Delta(\Lambda)\)).
\end{description}
Let \(m_{\Delta}\) be the Loewy length  of the mesh category \(k\mathbb Z\Delta\). From \cite{RmVenner} we know that \(m_{\A_n}=n,\  m_{\D_n}=2n-3,\  m_{\E_6}=11,\  m_{\E_7}=17\) and \(m_{\E_8}=29\).
The AR-quiver of the stable module category of \(\Lambda\) is known \cite{riedtmannTysk} to be on the form \(\mathbb Z\Delta/\langle\phi\tau^{-r}\rangle\) for some automorphism \(\phi\) with a fixed vertex.
\begin{description}[font=\bfseries]
\item[\(f(\Lambda)\)] the frequency of \(\Lambda\) is given by \(f(\Lambda):=r/m_{\Delta}\).
\item[\(t(\Lambda)\)] the torsion order \(t(\Lambda)\) is the order of \(\phi\). 
\end{description}
\end{definition}
Using this notation, Asashiba gives a list of the types a standard representation-finite self-injective indecomposable algebra can have.

\begin{theorem}[\cite{asashiba2}]\label{standardRFS}
The set of types of standard representation-finite self-injective indecomposable algebras is the disjoint union of the following sets:
\begin{itemize}
\item \(\left\{(\A_n,\frac{s}{n},1)|n,s\in\mathbb N\right\}\)
\item \(\left\{(\A_{2p+1},s,2)|n,s\in\mathbb N\right\}\)
\item \(\left\{(\D_n,s,1)|n,s\in\mathbb N,n\geq 4\right\}\)
\item \(\left\{(\D_{3m},\frac{s}{3},1)|m,s\in\mathbb N, m\geq 2, 3\nmid s \right\}\)
\item \(\left\{(\D_n,s,2)|n,s\in\mathbb N, n\geq 4\right\}\)
\item \(\left\{(\D_4,s,3)|s\in\mathbb N\right\}\)
\item \(\left\{(\E_n,s,1)|n=6,7,8, s\in\mathbb N\right\}\)
\item \(\left\{(\E_6,s,2)|s\in\mathbb N\right\}\)
\end{itemize}
\end{theorem}
Since this is an exhaustive list of all families of standard representation-finite self-injective algebras, this list also tells us what orbit categories to consider. If an orbit category has finitely many indecomposables, but is not of any of the types in the list, it cannot be equivalent to the stable module category of a self-injective algebra.

\section{Type \(\A\)}\label{typeA}
There are two standard types of representation-finite self-injective algebras that have AR-quivers of the form \(\mathbb Z\A_n/G\), up to stable equivalence. The representatives gives for these two standard types by \cite{asashiba1} and also by \cite{riedtmannA} are the Nakayama algebras, with AR-quivers of cylindrical shape, and the M\"obius algebras, which have AR-quivers shaped like a M\"obius band. 
The Nakayama algebra case was also considered in \cite[Thm.\ 3.3.8]{XiaoZhu}, but we will restate it for the sake of completeness.

For the Nakayama algebras, the stable module categories will be equivalent to orbit categories using functors that are some power of the AR-translation \(\tau\). For M\"obius algebras we need a "flip functor" to get the M\"obius shape of the quiver:
\begin{definition}\label{Aflip}
Let \(n=2l+1\) with \(l\in\mathbb N\). The flip functor \(\phi\) on \(\mathcal D^b(k\A_n)\) is given by \(\phi=\tau^{l+1} [1]\).
\end{definition}

\subsection{Self-injective Nakayama algebras}\label{Nakayama}

\begin{definition}
\begin{figure}[bht]
\centering
\begin{tikzpicture}[bend left=15,radius=2cm, auto]
\node (1) at (90:2){\(1\)};
\node (2) at (45:2){\(2\)};
\node (v) at (135:2){\(v\)};
\draw[->] (1) to node {\(\alpha_1\)} (2);
\draw[->] (v) to node {\(\alpha_v\)} (1);
\draw[->, dashed] (2) ..controls (0:2) and (180:2) .. (v);
\end{tikzpicture}\caption{Quiver of a self-injective Nakayama algebra  \(N_{v,r}\)}\label{NakayamaQ}
\end{figure}
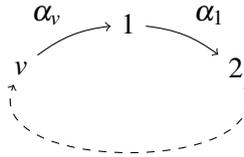
A self-injective Nakayama algebra is a path algebra \(N_{v,r}=Q_v/I_r\), for \(v\geq 1,r\geq 2\), where
\(Q_{v}\) is the quiver in figure \ref{NakayamaQ} and \({I_r}\) is the ideal generated by paths of length \(r\).
\end{definition}
These algebras are self-injective, and the stable module category \(\smod N_{v,r}\) is triangulated.
The AR-quiver of \(\smod N_{v,r}\) has been described by Riedtmann in \cite{riedtmannA}. As a translation quiver it is of the form \(\mathbb Z \A_{r-1}/(\theta^v)\). In the notation of Asashiba this is of type \((A_n,\frac{v}{r},1)\).

We denote the indecomposable modules over \(N_{v,r}\) by \(M_n^l\), where \(n\) is the socle of the module, and \(l\) is the (Loewy) length of the module. The AR-quiver of \(\smod N_{v,r}\) is shown in Figure \ref{NakayamaAR}.

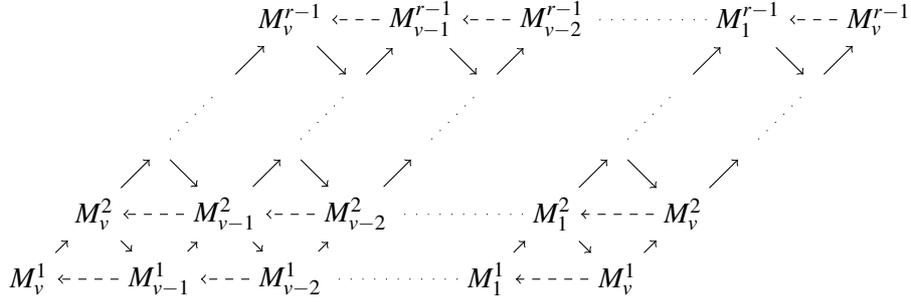
\begin{figure}[ht]
\centering
\begin{tikzpicture}[scale=.86]
\node (1v) at (0,0)   {\(M_{v}^1\)};
\node (1v-1) at (2,0) {\(M_{v-1}^1\)};
\node (1v-2) at (4,0) {\(M_{v-2}^1\)};
\node (12) at (7,0)   {\(M_{1}^1\)};
\node (11) at (9,0)   {\(M_{v}^1\)};

\node (2v) at (1,1)   {\(M_{v}^2\)};
\node (2v-1) at (3,1) {\(M_{v-1}^2\)};
\node (2v-2) at (5,1) {\(M_{v-2}^2\)};
\node (22) at (8,1)   {\(M_{1}^2\)};
\node (21) at (10,1)  {\(M_{v}^2\)};

\node (rv) at (4,4)   {\(M_{v}^{r-1}\)};
\node (rv-1) at (6,4) {\(M_{v-1}^{r-1}\)};
\node (rv-2) at (8,4) {\(M_{v-2}^{r-1}\)};
\node (r2) at (11,4)  {\(M_{1}^{r-1}\)};
\node (r1) at (13,4)  {\(M_{v}^{r-1}\)};
\node (3v) at (2,2)   {};
\node (3v-1) at (4,2) {};
\node (3v-2) at (6,2) {};
\node (32) at (9,2)   {};
\node (31) at (11,2)  {};

\node (r-1v) at (3,3)   {};
\node (r-1v-1) at (5,3) {};
\node (r-1v-2) at (7,3) {};
\node (r-12) at (10,3)  {};
\node (r-11) at (12,3)  {};
\draw[->] (1v)--(2v);
\draw[->] (1v-1)--(2v-1);
\draw[->] (1v-2)--(2v-2);
\draw[->] (12)--(22);
\draw[->] (11)--(21);

\draw[->] (2v)--(3v);
\draw[->] (2v-1)--(3v-1);
\draw[->] (2v-2)--(3v-2);
\draw[->] (22)--(32);
\draw[->] (21)--(31);

\draw[->] (r-1v)--(rv);
\draw[->] (r-1v-1)--(rv-1);
\draw[->] (r-1v-2)--(rv-2);
\draw[->] (r-12)--(r2);
\draw[->] (r-11)--(r1);

\draw[->] (2v)--(1v-1);
\draw[->] (2v-1)--(1v-2);
\draw[->] (22)--(11);

\draw[->] (3v)--(2v-1);
\draw[->] (3v-1)--(2v-2);
\draw[->] (32)--(21);

\draw[->] (rv)--(r-1v-1);
\draw[->] (rv-1)--(r-1v-2);
\draw[->] (r2)--(r-11);
\draw[loosely dotted] (3v)--(r-1v);
\draw[loosely dotted] (3v-1)--(r-1v-1);
\draw[loosely dotted] (3v-2)--(r-1v-2);
\draw[loosely dotted] (32)--(r-12);
\draw[loosely dotted] (31)--(r-11);
\draw[loosely dotted] (12) -- (1v-2);
\draw[loosely dotted] (22) -- (2v-2);
\draw[loosely dotted] (r2) -- (rv-2);
\draw[dashed, ->] (1v-1) -- (1v);
\draw[dashed, ->] (2v-1) -- (2v);
\draw[dashed, ->] (rv-1) -- (rv);

\draw[dashed, ->] (1v-2) -- (1v-1);
\draw[dashed, ->] (2v-2) -- (2v-1);
\draw[dashed, ->] (rv-2) -- (rv-1);

\draw[dashed, ->] (11) -- (12);
\draw[dashed, ->] (21) -- (22);
\draw[dashed, ->] (r1) -- (r2);
\end{tikzpicture}
\caption{AR-quiver of  \(\smod N_{v,r}\).
The leftmost and rightmost diagonal are identified.}
\label{NakayamaAR}
\end{figure}

\begin{proposition}[{\cite[Thm.\ 3.3.8]{XiaoZhu}}]
The categories \(\smod N_{v,r}\)  and \(\mathcal D^b(k\A_{r-1})/\tau^v\) are triangle equivalent for \(r\geq2\) and \(v\in\mathbb N\setminus\{0\}\).
\end{proposition}

\begin{proof}
For \(v\neq 0\), the functor \(\tau^v\) fulfills the conditions in theorem \ref{KellerTriOrb}, so \(\mathcal D^b\!(k\A_{r-1}\!)/\!\tau^v\) is triangulated. The algebra \(N_{v,r}\) is a representation-finite, self-injective, basic algebra, whose stable module category is standard by theorem \ref{standardRFS}. 

Consider the stable translation quiver \((\mathbb Z\A_{r-1}, \theta)/\langle \theta^v \rangle\). This is the quiver we get if we take the quiver \((\mathbb Z\A_{r-1}, \theta)\) from Figure \ref{translationq} and identify \((p,q)\) with \((p+v,q)\) for all \((q\in \mathbb Z)\). One morphism of translation quivers from the AR-quiver of the algebra \(N_{v,r}\) to \((\mathbb Z\A_{r-1}, \theta)/\langle \theta^v \rangle\) is given by \(M_p^q\mapsto (v-p, q)\) (the map on the arrows follow uniquely). Moreover, it is an isomorphism with inverse given by  \((p, q)\mapsto M_{v-p}^q\)

The conclusion follows from Corollary \ref{sameARquiver}.
\end{proof}
The explicit translation quiver isomorphisms can be found in a similar way in the other cases.

\subsection{M\"{o}bius algebras}\label{Mobius}
\begin{figure}[ht]
\centering
\begin{tikzpicture}[font=\footnotesize, scale=1.1]
\node (v1)  at (1,1) {\(\circ\)};
\node (v1a) at (2,1) {\(\circ\)};
\node (v1b) at (2,0) {\(\circ\)};
\node (v1c) at (3,1) {};
\node (v1d) at (3,0) {};
\node (v1e) at (4,1) {};
\node (v1f) at (4,0) {};
\node (v1g) at (5,1) {\(\circ\)};
\node (v1h) at (5,0) {\(\circ\)};

\node (v2)  at (6,1) {\(\circ\)};
\node (v2a) at (6,2) {\(\circ\)};
\node (v2b) at (7,2) {\(\circ\)};
\node (v2c) at (6,3) {};
\node (v2d) at (7,3) {};

\node (vxe) at (3,6) {};
\node (vxf) at (3,7) {};
\node (vxg) at (2,6) {\(\circ\)};
\node (vxh) at (2,7) {\(\circ\)};

\node (vv)  at (1,6) {\(\circ\)};
\node (vva) at (1,5) {\(\circ\)};
\node (vvb) at (0,5) {\(\circ\)};
\node (vvc) at (1,4) {};
\node (vvd) at (0,4) {};
\node (vve) at (1,3) {};
\node (vvf) at (0,3) {};
\node (vvg) at (1,2) {\(\circ\)};
\node (vvh) at (0,2) {\(\circ\)};

\draw[->] (v1)-- node[above]{\(\alpha_0^1\)}(v1a);
\draw[->] (v1)-- node[below]{\(\beta_0^1\)}(v1b);
\draw[->] (v1a)-- node[above]{\(\alpha_1^1\)}(v1c);
\draw[->] (v1b)-- node[below]{\(\beta_1^1\)}(v1d);
\draw[dotted] (v1c)-- node[above]{}(v1e);
\draw[dotted] (v1d)-- node[below]{}(v1f);
\draw[->] (v1e)-- node[above]{\(\alpha_{l-1}^1\)}(v1g);
\draw[->] (v1f)-- node[below]{\(\beta_{l-1}^1\)}(v1h);
\draw[->] (v1g)-- node[above]{\(\alpha_{l}^1\)}(v2);
\draw[->] (v1h)-- node[below]{\(\beta_{l}^1\)}(v2);

\draw[->] (v2)-- node[above right]{\(\alpha_0^2\)}(v2a);
\draw[->] (v2)-- node[right]{\(\beta_0^2\)}(v2b);
\draw[->] (v2a)-- node[right]{\(\alpha_1^2\)}(v2c);
\draw[->] (v2b)-- node[right]{\(\beta_1^2\)}(v2d);

\draw[dotted] (v2c) .. controls (6,5) and (5,6) .. (vxe);
\draw[dotted] (v2d) .. controls (7,6) and (6,7) .. (vxf);

\draw[->] (vxe)-- node[above]{\(\alpha_{l-1}^{v-1}\)}(vxg);
\draw[->] (vxf)-- node[above]{\(\beta_{l-1}^{v-1}\)}(vxh);
\draw[->] (vxg)-- node[above right]{\(\alpha_{l}^{v-1}\)}(vv);
\draw[->] (vxh)-- node[above]{\(\beta_{l}^{v-1}\)}(vv);

\draw[->] (vv)-- node[right]{\(\alpha_0^v\)}(vva);
\draw[->] (vv)-- node[left]{\(\beta_0^v\)}(vvb);
\draw[->] (vva)-- node[right]{\(\alpha_1^v\)}(vvc);
\draw[->] (vvb)-- node[left]{\(\beta_1^v\)}(vvd);
\draw[dotted] (vvc)-- node[right]{}(vve);
\draw[dotted] (vvd)-- node[left]{}(vvf);
\draw[->] (vve)-- node[right]{\(\alpha_{l-1}^v\)}(vvg);
\draw[->] (vvf)-- node[left]{\(\beta_{l-1}^v\)}(vvh);
\draw[->] (vvg)-- node[right]{\(\alpha_{l}^v\)}(v1);
\draw[->] (vvh)-- node[left]{\(\beta_{l}^v\)}(v1);
\end{tikzpicture}
\caption{Quiver of the M\"obius algebra \(M_{l,v}\)}
\label{MobiusQ}
\end{figure}
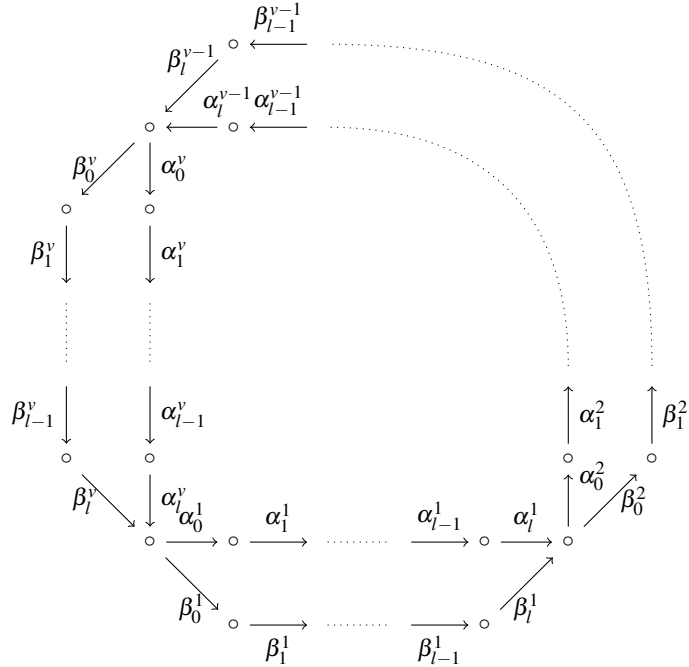
\begin{definition}
Let \(l,v\geq 1\).
The M\"obius algebra \(M_{l,v}\) is the path algebra \(kQ/I\), where \(Q\) is the quiver in figure \ref{MobiusQ}
and \(I\) is generated by the relations:
\begin{enumerate}
\item \(\alpha^i_l\cdots\alpha^i_0=\beta^i_l\cdots\beta^i_0\) for  \(i\in\{1,\ldots, v\}\)
\item \(\beta^{i+1}_0\alpha^i_l=0\) and \(\alpha^{i+1}_0\beta^i_l=0\) for \(i\in\left\{1,\ldots, v-1\right\}\)
\item \(\alpha^{1}_0\alpha^v_l=0\) and \(\beta^{1}_0\beta^v_l=0\)
\item paths of length \(l+2\) are equal to zero
\end{enumerate}
\end{definition}

\begin{example}
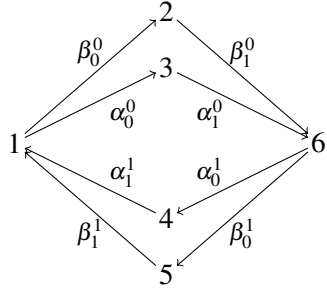
\begin{figure}[ht]
\centering
\begin{tikzpicture}
\node(1) at (0,0)[inner sep=1 pt]{\(1\)};
\node(2) at (2,1.75)[inner sep=1 pt]{\(2\)};
\node(3) at (2,1)[inner sep=1 pt]{\(3\)};
\node(4) at (2,-1)[inner sep=1 pt]{\(4\)};
\node(5) at (2,-1.75)[inner sep=1 pt]{\(5\)};
\node(6) at (4,0)[inner sep=1 pt]{\(6\)};

\draw[->](1)-- node[above]{\footnotesize \(\beta^0_0\)}(2);
\draw[->](2)-- node[above]{\footnotesize \(\beta^0_1\)}(6);
\draw[->](1)-- node[below, near end]{\footnotesize \(\alpha^0_0\)}(3);
\draw[->](3)-- node[below, near start]{\footnotesize \(\alpha^0_1\)}(6);
\draw[->](6)-- node[above, near end]{\footnotesize \(\alpha^1_0\)}(4);
\draw[->](4)-- node[above, near start]{\footnotesize \(\alpha^1_1\)}(1);
\draw[->](6)-- node[below]{\footnotesize \(\beta^1_0\)}(5);
\draw[->](5)-- node[below]{\footnotesize \(\beta^1_1\)}(1);

\end{tikzpicture}
\caption{The quiver of \(M_{1,2}\)}
\label{mobiusExQ}
\end{figure}

Let \(l=1\) and \(v=2\). The algebra \(M_{1,2}\) is given by the quiver in Figure \ref{mobiusExQ} with relations
\begin{align*}
\alpha_1^0\alpha_0^0&=\beta_1^0\beta_0^0 &\alpha_1^1\alpha_0^1&=\beta_1^1\beta_0^1\\
\beta_0^1\alpha_1^0&=0 & \alpha_0^1\beta_1^0&=0\\
\alpha_0^0\alpha_1^1&=0 & \beta_0^0\beta_1^1&=0.\\
\end{align*}
The AR-quiver of this algebra is shown in Figure \ref{ARquiverMoebius}. We see that \(\smod M_{1,2}\) is triangle equivalent to \(\Db(k \A_3)/\phi\tau^6\).

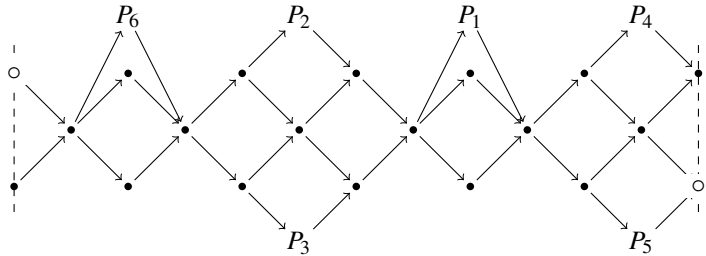
\begin{figure}[ht]
\centering
\begin{tikzpicture}[scale=.75,yscale=-1]
 \foreach \x in {0,...,6}
  \foreach \y in {1,3}
   \node (\y-\x) at (\x*2,\y) [vertex] {};
 \foreach \x in {0,...,5}
  \foreach \y in {2}
   \node (\y-\x) at (\x*2+1,\y) [vertex] {};
 \foreach \xa/\xb in {0/1,1/2,2/3,3/4,4/5,5/6}
  \foreach \ya/\yb in {1/2,3/2}
   {
    \draw [->] (\ya-\xa) -- (\yb-\xa);
    \draw [->] (\yb-\xa) -- (\ya-\xb);
   }  
 \draw [dashed] (0,.5) -- (0,3.5); 
 \draw [dashed] (12,.5) -- (12,3.5);
 
 \node (P6) at (2,0)[inner sep=1 pt]{\small \(P_6\)};
 \draw [->] (2-0)--(P6);
 \draw[->] (P6)--(2-1);
 \node(P2) at (5,0)[inner sep=1 pt]{\small \(P_2\)};
 \draw[->] (1-2) -- (P2);
 \draw[->] (P2) -- (1-3);
 \node(P3) at (5,4)[inner sep=1 pt]{\small \(P_3\)};
 \draw[->] (3-2)--(P3);
 \draw[->] (P3)--(3-3);
 \node(P1) at (8,0)[inner sep=1 pt]{\small \(P_1\)};
 \draw[->](2-3)--(P1);
 \draw[->](P1)--(2-4);
 \node(P4) at (11,0)[inner sep=1 pt]{\small \(P_4\)};
 \draw[->](1-5)--(P4);
 \draw[->](P4) --(1-6);
 \node(P5) at (11,4)[inner sep=1 pt]{\small \(P_5\)};
 \draw[->](3-5)--(P5);
 \draw[->](P5)--(3-6);
 \replacevertex {(1-0)}{[fill=white, inner sep=2pt]{\(\circ\)}}
 \replacevertex {(3-6)}{[fill=white, inner sep=2pt]{\(\circ\)}}
 
\end{tikzpicture} \caption{The AR-quiver of the algebra \(M_{1,2}\). The identical objects on either side are identified.}\label{ARquiverMoebius}
\end{figure}
\end{example}

Riedtmann\cite{riedtmannA} showed that in general the AR-quiver of the stable module category of a M\"obius algebra \(M_{l,v}\) is of the form \(\mathbb Z\A_{2l+1}/(\theta^{(2l+1)v}\phi)\), where \(\phi=\theta^\frac{2l+2}{2}S\) and \(S\) is as in table \ref{tableOfS}. It is the asymmetry of relations \((2)\) and \((3)\) in \(I\) that gives rise to the "M\"obius" twist.

In Asashiba's notation these algebras are of type \((A_{2l+1},v,2)\).

\begin{proposition}
Let \(l, v\!\geq 1\) and let \(n\!=\!2l\!+1\!\). The categories \(\smod M_{l,v}\) and \(D^b(k\!\A_n\!)/\tau^{nv}\!\phi\) are equivalent as triangulated categories.
\end{proposition}
\begin{proof}
Since \(nv\geq 1\), we know that \(\tau^{nv}\phi\) fulfils the requirements on \(F\) in theorem \ref{KellerTriOrb}. Hence  \(D^b(k\A_{2l+1})/\tau^{nv}\phi\) is triangulated. The algebra \(M_{l,v}\) is a representation-finite, self-injective, basic algebra, whose stable module category is standard by theorem \ref{standardRFS}. The conclusion follows from Corollary \ref{sameARquiver}.
\end{proof}

\section{Type \(\D\)}\label{typeD}

We will now look in detail at the classes of self-injective algebras that have AR-quivers of the form \(\mathbb Z\D_n/G\). For this purpose we will make use of the detailed list of representatives of the standard types of representation-finite self-injective algebras provided as an appendix to \cite{asashiba1}. There are, as indicated by theorem \ref{standardRFS}, four cases to consider that are standard. Three of these share the same quiver but have different sets of relations, the last type has an entirely different quiver. 

\begin{figure}[bh]
\centering
\begin{tikzpicture}[font=\footnotesize, scale=1.1]
\node(g1) at (-1,1){\(\circ\)}; 
\node(g2) at (-1,0){\(\circ\)};
\node(g3) at (-1,-1){\(\circ\)};
\node(g4) at (0,-1){\(\circ\)};
\node(g5) at (1,-1){\(\circ\)}; 
\node(g6) at (1,0){};
\node(g7) at (1,0.7){};
\node(g8) at (0,1){};
\node(g9) at (0.7,1){};
\draw[->](g8)--node[above right]{\(\gamma^{s-1}_1\)}(g1);
\draw[->](g1)--node[right]{\(\gamma^0_0\)}(g2);
\draw[->](g2)--node[right]{\(\gamma^0_1\)}(g3);
\draw[->](g3)--node[above right]{\(\gamma^1_0\)}(g4);
\draw[->](g4)--node[above left]{\(\gamma^1_1\)}(g5);
\draw[->](g5)--node[left]{\(\gamma^2_0\)}(g6);
\draw[dotted](g9)--(g8);
\draw[dotted](g6)--(g7);
\node(b1) at (-2,0){\(\circ\)};
\node(b2) at (0,-2){\(\circ\)};
\node(b3) at (2,0){};
\node(b4) at (0,2){};
\node(b5) at (1.3,0.7){};
\node(b6) at (0.7,1.3){};
\draw[->](g1)--node[above, very near end]{\(\beta^0_0\)}(b1);
\draw[->](b1)--node[below, very near start]{\(\beta^0_1\)}(g3);
\draw[->](g3)--node[below, midway]{\(\beta^1_0\)}(b2);
\draw[->](b2)--node[below, midway]{\(\beta^1_1\)}(g5);
\draw[->](g5)--node[below,near end]{\(\beta^0_2\)}(b3);
\draw[->](b4)--node[above]{\(\beta^{s-1}_1\)}(g1);
\draw[dotted](b4)--(b6);
\draw[dotted] (b3)--(b5);
\node(a1) at (-3,1) {\(\circ\)};
\node(a2) at (-3,0.25){};
\node(a3) at (-3,-0.25){};
\node(a4) at (-3,-1){\(\circ\)};
\node(a5) at (-1,-3) {\(\circ\)};
\node(a6) at (-0.25,-3){};
\node(a7) at (0.25,-3){};
\node(a8) at (1,-3){\(\circ\)};
\node(a9) at (3,-1) {\(\circ\)};
\node(a10) at (3,-0.25){};
\node(a11) at (3,0.25){};
\node(a12) at (0.25,3){};
\node(a13) at (-0.25,3){};
\node(a14) at (-1,3){\(\circ\)};
\node(a15) at (1.25,2){};
\node(a16) at (2,1.25){};
\draw[->](g1)--node[above]{\(\alpha^0_{n-2}\)}(a1);
\draw[->](a1)--node[left]{\(\alpha^0_{n-3}\)}(a2);
\draw[dotted](a2)--(a3);
\draw[->](a3)--node[left]{\(\alpha^0_2\)}(a4);
\draw[->](a4)--node[below]{\(\alpha^0_1\)}(g3);
\draw[->](g3)--node[left]{\(\alpha^1_{n-2}\)}(a5);
\draw[->](a5)--node[below]{\(\alpha^1_{n-3}\)}(a6);
\draw[dotted](a6)--(a7);
\draw[->](a7)--node[below]{\(\alpha^1_{2}\)}(a8);
\draw[->](a8)--node[right]{\(\alpha^1_1\)}(g5);

\draw[->](g5)--node[below]{\(\alpha^2_{n-2}\)}(a9);
\draw[->](a9)--node[right]{\(\alpha^2_{n-3}\)}(a10);
\draw[dotted](a10)--(a11);
\draw[dotted](a12)--(a13);
\draw[->](a13)--node[above]{\(\alpha^{s-1}_{2}\)}(a14);
\draw[->](a14)--node[left]{\(\alpha^{s-1}_{1}\)}(g1);
\draw[dotted](a15)--(a16);
\end{tikzpicture}\caption{\((\D_n,s)\)}\label{koggerDhoved}
\end{figure}

We will now define some  automorphisms that induce functors that will be useful in later subsections. Recall that by Theorem \ref{Picard}, an automorphism on the translation quiver \((\Z \D_n, \theta)\) induces a functor on \(\Db(k\D_n\)). The definitions are given with respect to the translation quiver shown in Figure \ref{translationq}.


\begin{definition}\label{Dflip}
Let \(\xi\) be the automorphism on \((\Z \D_n,\theta)\), which exchanges the vertices \((x, n)\) and \((x, n-1)\).
 Let \(\phi\) be the autoequivalence on \(\Db(k\D_n)\) induced by \(\xi\).

%
\end{definition}

\begin{definition}\label{Dtwist}
Let \(\chi\) be the automorphism on \((\Z \D_4, \theta)\) which acts as follows:
\[
\begin{array}{c|cccc}
(x,r)		& (x,1) & (x,2) & (x,3) & (x,4)	\\\hline
\chi(x,r)	& (x-1,3) & (x,2) & (x,4) & (x+1,1)
\end{array}
\text{ for all \(x\in\Z\)}.
\]
 Let \(\rho\) be the autoequivalence on \(\Db(k\D_4)\) induced by \(\chi\).

\end{definition}

\subsection{Type \((\D_n,s,1)\)}\label{Drot1}

\begin{definition}
The representative of self-injective algebras of type \((\D_n,s,1)\) is given by the path algebra \(D_{n,s,1}:=kQ/I\) where \(Q\) is the quiver of figure \ref{koggerDhoved} and the ideal \(I\) is generated by the following set of relations:
\begin{enumerate}
\item \label{rel1}\(\alpha^i_1\alpha^i_2\cdots\alpha^i_{n-2}=\beta^i_1\beta^i_0=\gamma^i_1\gamma^i_0\) for all \(i\in\{0,\ldots,s-1\}\)
\item\label{rel2} For all \(i\in\{0,\ldots,s-1\}=\mathbb Z/\langle s\rangle,\)
\begin{align*}
&\beta^{i+1}_0\alpha^i_1=0,\;&\gamma^{i+1}_0\alpha^i_1=0,\\
& \alpha^{i+1}_{n-2}\beta^i_1=0,\;&\alpha^{i+1}_{n-2}\gamma^i_1=0,\\
&\gamma^{i+1}_0\beta^i_1=0,\;&\beta^{i+1}_0\gamma^i_1=0;
\end{align*}
\item \(\alpha^{i+1}_{j-n+2}\cdots\alpha^i_j=0\) for all \(i\in\{0,\cdots\!,s-1\}=\mathbb Z/\!\langle s\rangle\) and for all \(j\in\{1,\cdots\!,n-2\}=\mathbb Z/\langle n-2\rangle\).
\end{enumerate}
\end{definition}

\begin{example}

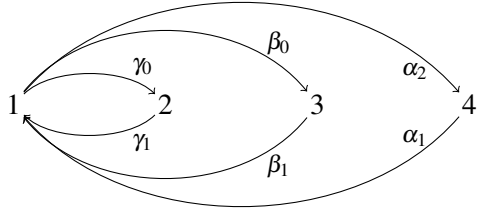
\begin{figure}[th]
\centering
\begin{tikzpicture}
\node(1) at (0,0)[inner sep=1 pt]{\(1\)};
\node(2) at (2,0)[inner sep=1 pt]{\(2\)};
\node(3) at (4,0)[inner sep=1 pt]{\(3\)};
\node(4) at (6,0)[inner sep=1 pt]{\(4\)};

\draw[->] (1) .. controls (0.5,0.5) and (1.5,.5)  .. node[above, very near end]{\footnotesize \(\gamma_0\)}(2);
\draw[->] (2) .. controls (1.5,-0.5) and (.5,-.5)  .. node[below,very  near start]{\footnotesize \(\gamma_1\)}(1);
\draw[->] (1) .. controls (1,1.25) and (3,1.25) .. node[above, very near end]{\footnotesize\(\beta_0\)}(3);
\draw[->] (3) .. controls (3,-1.25) and (1,-1.25) .. node[below, very near start]{\footnotesize\(\beta_1\)}(1);
\draw[->] (1) .. controls (1.5,1.75) and (4.5,1.75) .. node[below, very near end]{\footnotesize \(\alpha_2\)}(4);
\draw[->] (4) .. controls (4.5,-1.75) and (1.5,-1.75) .. node[above, very near start]{\footnotesize \(\alpha_1\)}(1);
\end{tikzpicture}\caption{The quiver of algebras \(D_{4,1,1},D_{4,1,2}\) and \(D_{4,1,3}\).}\label{D_{4,1,r}kogger}
\end{figure}

Let \(n=4\) and \(s=1\). The algebra \(D_{4,1,1}\) is given by the quiver in figure \ref{D_{4,1,r}kogger} with relations:
\[\alpha_1\alpha_2=\beta_1\beta_0=\gamma_1\gamma_0\]
\begin{align*}
&\alpha_2\beta_1=0\;,&&\beta_0\alpha_1=0\;,&& \gamma_0\alpha_1=0,\\
&\alpha_2\gamma_1=0\;,&&\beta_0\gamma_1=0\;,&&\gamma_0\beta_1=0,
\end{align*}
and all paths of length \(3\) are \(0\). Note that the relations in point \ref{rel2} makes it impossible to compose arrows from different loops, this leads to an AR-quiver which has cylinder shape. The AR-quiver of this algebra is shown in figure \ref{ARD_{4,1,1}}. In this case \(\smod D_{4,1,1}\) is triangle equivalent to \(\mathcal D^b(k\D_4)/\tau^5\).

\begin{figure}[ht]
\centering
 \begin{tikzpicture}[scale=0.75,yscale=-1]
 \foreach \x in {0,...,5}
  \foreach \y in {1,2,4,}
   \node (\y-\x) at (\x*2,\y) [vertex] {};
 \foreach \x in {0,...,4}
  \foreach \y in {3}
   \node (\y-\x) at (\x*2+1,\y) [vertex] {};
 \foreach \xa/\xb in {0/1,1/2,2/3,3/4,4/5}
  \foreach \ya/\yb in {1/3,2/3,4/3}
   {
    \draw [->] (\ya-\xa) -- (\yb-\xa);
    \draw [->] (\yb-\xa) -- (\ya-\xb);
   }
 \draw [dashed] (0,.5) -- (0,4.5); 
 \draw [dashed] (10,.5) -- (10,4.5);
\node(P1) at (2,0)[inner sep=1 pt]{\small \(P_1\)};
\draw[->](3-0)--(P1);
\draw[->](P1)--(3-1);
\node(P2) at (7,0)[inner sep=1 pt]{\small \(P_2\)};
\draw[->](1-3)--(P2);
\draw[->](P2)--(1-4);
\node(P3) at (7,1)[inner sep=1 pt]{\small \(P_3\)};
\draw[->](2-3)--(P3);
\draw[->](P3)--(2-4);
\node(P4) at (7,5)[inner sep=1 pt]{\small \(P_4\)};
\draw[->](4-3)--(P4);
\draw[->](P4)--(4-4);   
\end{tikzpicture}\caption{\(D_{4,1,1}\)}\label{ARD_{4,1,1}}
\end{figure}
\end{example}

The AR-quiver of the stable module category of algebras of type \((\D_n,s,1)\) is of the form \(\mathbb Z\D_n/\theta^{s(h-1)}\), where \(h\) is the Coxeter number for \(\D_n\). 

\begin{proposition}
Let \(n\geq 4\) and \(n,s\in\mathbb N\). The categories \(\smod D_{n,s,1}\) and \(\mathcal D^b(k\D_n)/\tau^{s(h-1)}\) are equivalent as triangulated categories.
\end{proposition}

\begin{proof}
Since \(s(h-1)>0\) the functor \(\tau^{s(h-1)}\) satisfies the conditions of theorem \ref{KellerTriOrb}, so the category \(\mathcal D^b(k\D_n)/\tau^{s(h-1)}\) is triangulated. 
The algebra \(D_{n,s,1}\) is a representation-finite, self-injective, basic algebra, whose stable module category is standard by theorem \ref{standardRFS}. The conclusion follows from Corollary \ref{sameARquiver}.
\end{proof}

\subsection{Type \((\D_n,s,2)\)}\label{Drot2}

\begin{definition}
The representative of self-injective algebras of type \((\D_n,s,2)\) is given by the path algebra \(D_{n,s,2}:=kQ/I\) where \(Q\) is the quiver of figure \ref{koggerDhoved} and the ideal \(I\) is generated by the following set of relations:
\begin{enumerate}
\item \(\alpha^i_1\alpha^i_2\cdots\alpha^i_{n-2}=\beta^i_1\beta^i_0=\gamma^i_1\gamma^i_0\) for all \(i\in\{0,\ldots,s-1\}\)
\item\label{rel22} for all \(i\in\{0,\ldots,s-1\}=\mathbb Z/\langle s\rangle\),
\begin{align*}
\beta^{i+1}_0\alpha^i_1&=0\; &\gamma^{i+1}_0\alpha^i_1=0,\\
\alpha^{i+1}_{n-2}\beta^i_1&=0\; &\alpha^{i+1}_{n-2}\gamma^i_1=0,
\end{align*}
and for all \(i\in\{0,\ldots,s-2\}\),
\begin{align*}
\gamma^{i+1}_0\beta^i_1&=0\; &\beta^{i+1}_0\gamma^i_1=0,\\
\beta^0_0\beta^{s-1}_1&=0,\; &\gamma^0_0\gamma^{s-1}_1=0;
\end{align*}
\item \(\alpha\)-paths of length \(n-1\) are zero, and for all \(i\in\{0,\ldots,s-2\}\),
\begin{align*}
\beta^{i+1}_0\beta^i_1\beta^i_0&=0,\; &\gamma^{i+1}_0\gamma^i_1\gamma^i_0&=0,\\
\beta^{i+1}_1\beta^{i+1}_0\beta^i_1&=0,\; & \gamma^{i+1}_1\gamma^{i+1}_0\gamma^i_1&=0, \text{ and}\\
\gamma^0_0\beta^{s-1}_1\beta^{s-1}_0&=0,\; &\beta^0_0\gamma^{s-1}_1\gamma^{s-1}_0&=0,\\
\gamma^0_1\gamma^0_0\beta^{s-1}_1&=0,\; &\beta^0_1\beta^0_0\gamma^{s-1}_1&=0.
\end{align*}
\end{enumerate}
\end{definition}

\begin{example}

Let \(n=4\) and \(s=1\). The algebra \(D_{4,1,2}\) is given by the quiver in figure \ref{D_{4,1,r}kogger} with relations:
\[\alpha_1\alpha_2=\beta_1\beta_0=\gamma_1\gamma_0\]
\begin{align*}
&\alpha_2\beta_1=0\;,&&\beta_0\alpha_1=0\;,&& \gamma_0\alpha_1=0,\\
&\alpha_2\gamma_1=0\;,&&\beta_0\beta_1=0\;,&& \gamma_0\gamma_1=0,
\end{align*}
and all paths of length \(3\) are \(0\). The AR-quiver of this algebra is shown in figure \ref{ARD_{4,1,2}}. This time the zero relations in point \ref{rel22} glues together two of the \(\tau\)-orbits of \(\mathbb Z\D_4\). In this case \(\smod D_{4,1,2}\) is triangle equivalent to \(\mathcal D^b(k\D_4)/\tau^5\phi\).

\begin{figure}[ht]
\centering
 \begin{tikzpicture}[scale=0.75,yscale=-1]
 \foreach \x in {0,...,5}
  \foreach \y in {1,2,4,}
   \node (\y-\x) at (\x*2,\y) [vertex] {};
 \foreach \x in {0,...,4}
  \foreach \y in {3}
   \node (\y-\x) at (\x*2+1,\y) [vertex] {};
 \foreach \xa/\xb in {0/1,1/2,2/3,3/4,4/5}
  \foreach \ya/\yb in {1/3,2/3,4/3}
   {
    \draw [->] (\ya-\xa) -- (\yb-\xa);
    \draw [->] (\yb-\xa) -- (\ya-\xb);
   }
 \draw [dashed] (0,.5) -- (0,4.5); 
 \draw [dashed] (10,.5) -- (10,4.5);
\node(P1) at (2,0)[inner sep=1 pt]{\small \(P_1\)};
\draw[->](3-0)--(P1);
\draw[->](P1)--(3-1);
\node(P2) at (7,0)[inner sep=1 pt]{\small \(P_2\)};
\draw[->](1-3)--(P2);
\draw[->](P2)--(1-4);
\node(P3) at (7,1)[inner sep=1 pt]{\small \(P_3\)};
\draw[->](2-3)--(P3);
\draw[->](P3)--(2-4);
\node(P4) at (7,5)[inner sep=1 pt]{\small \(P_4\)};
\draw[->](4-3)--(P4);
\draw[->](P4)--(4-4); 
\replacevertex {(1-0)}{[fill=white, inner sep=2pt]{\(\star\)}}
\replacevertex {(2-5)}{[fill=white, inner sep=2pt]{\(\star\)}}
\replacevertex{(2-0)} {[fill=white, inner sep=2pt]{\(\circ\)}}
\replacevertex{(1-5)} {[fill=white, inner sep=2pt]{\(\circ\)}}
  
\end{tikzpicture}\caption{\(\modf D_{4,1,2}\). The quiver is glued together by identifying the matching symbols on either side.}\label{ARD_{4,1,2}}
\end{figure}
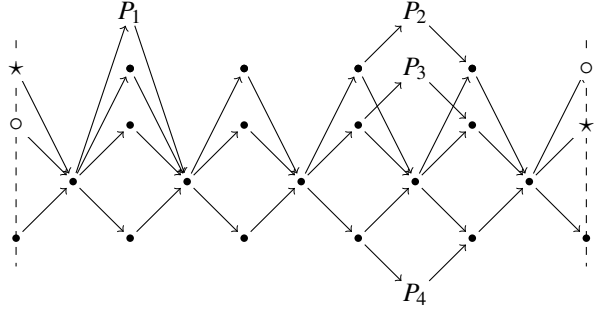

\end{example}

The AR-quiver of the stable module category of algebras of type \((\D_n,s,2)\) is of the form \(\mathbb Z\D_n/\theta^{s(h-1)}\xi\), where \(h\) is the Coxeter number for \(\D_n\), and \(\xi\) is the automorphism described in Definition \ref{Dflip}.

\begin{proposition}
Let \(n\leq4\) and \(s,n\!\in\mathbb N\). The categories \(\smod D_{n,s,2}\) and \linebreak\(\mathcal D^b(k\D_n)/\tau^{s(h-1)}\phi\) are equivalent as triangulated categories.
\end{proposition}

\begin{proof}
Since \(s(h-1)>0\) the functor \(\tau^{s(h-1)}\phi\) satisfies the conditions given in theorem \ref{KellerTriOrb}. Hence the category \(\mathcal D^b(k\D_n)/\tau^{s(h-1)}\phi\) is triangulated. 
The algebra \(D_{n,s,2}\) is a representation-finite, self-injective, basic algebra, whose stable module category is standard by theorem \ref{standardRFS}. The conclusion follows from Corollary \ref{sameARquiver}.
\end{proof}

\subsection{Type \((\D_4,s,3)\)}\label{Drot3}

\begin{definition}
The representative of self-injective algebras of type \((\D_4,s,3)\) is given by the path algebra \(D_{4,s,3}:=kQ/I\) where \(Q\) is the quiver of figure \ref{koggerDhoved} and the ideal \(I\) is generated by the following set of relations:
\begin{enumerate}
\item The same relations as for \((\D_4,s,1)\), part \ref{rel1}.
\item For all \(i\in\{0,\ldots,s-2\}\)
\begin{align*}
\beta^{i+1}_0\alpha^i_1&=0,\; & \gamma^{i+1}_0\alpha^i_1&=0,\\
\alpha^{i+1}_0\beta^i_1&=0,\; & \gamma^{i+1}_0\beta^i_1&=0,\\
\alpha^{i+1}_0\gamma^i_1&=0,\; & \beta^{i+1}_0\gamma^i_1&=0,\text{ and}\\
\alpha^0_0\alpha^{s-1}_1&=0,\; & \gamma^0_0\alpha^{s-1}_1&=0,\\
\alpha^0_0\beta^{s-1}_1&=0,\; & \beta^0_0\beta^{s-1}_1&=0,\\
\beta^0_0\gamma^{s-1}_1&=0,\; & \gamma^0_0\gamma^{s-1}_1&=0;
\end{align*}
\item all paths of length \(3\) are zero.
\end{enumerate}
\end{definition}

\begin{example}

Let \(n=4\) and \(s=1\). The algebra \(D_{4,1,3}\) is given by the quiver in figure \ref{D_{4,1,r}kogger} with relations:
\[\alpha_1\alpha_2=\beta_1\beta_0=\gamma_1\gamma_0\]
\begin{align*}
&\alpha_0\alpha_1=0\;,&&\alpha_0\beta_1=0\;,&& \beta_0\gamma_1=0,\\
&\gamma_0\alpha_1=0\;,&&\beta_0\beta_1=0\;,&& \gamma_0\gamma_1=0,
\end{align*}
and all paths of length \(3\) are \(0\). The AR-quiver of this algebra is shown in figure \ref{ARD_{4,1,3}}. As the figure shows, three of the \(\tau\)-orbits of \(\mathbb Z\D_4\) are glued together, this is due to the zero relations of length two. In this case \(\smod D_{4,1,3}\) is triangle equivalent to \(\mathcal D^b(k\D_4)/\tau^5\rho\).

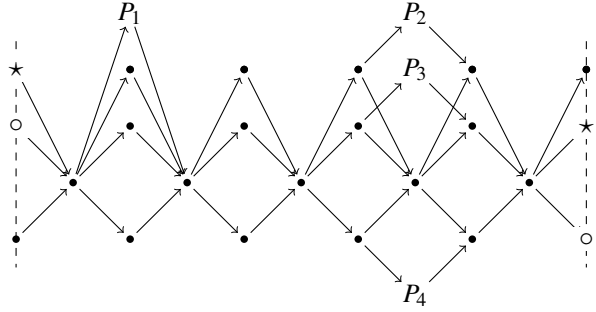
\begin{figure}[ht]
\centering
 \begin{tikzpicture}[scale=0.75,yscale=-1]
 \foreach \x in {0,...,5}
  \foreach \y in {1,2,4,}
   \node (\y-\x) at (\x*2,\y) [vertex] {};
 \foreach \x in {0,...,4}
  \foreach \y in {3}
   \node (\y-\x) at (\x*2+1,\y) [vertex] {};
 \foreach \xa/\xb in {0/1,1/2,2/3,3/4,4/5}
  \foreach \ya/\yb in {1/3,2/3,4/3}
   {
    \draw [->] (\ya-\xa) -- (\yb-\xa);
    \draw [->] (\yb-\xa) -- (\ya-\xb);
   }
 \draw [dashed] (0,.5) -- (0,4.5); 
 \draw [dashed] (10,.5) -- (10,4.5);
\node(P1) at (2,0)[inner sep=1 pt]{\small \(P_1\)};
\draw[->](3-0)--(P1);
\draw[->](P1)--(3-1);
\node(P2) at (7,0)[inner sep=1 pt]{\small \(P_2\)};
\draw[->](1-3)--(P2);
\draw[->](P2)--(1-4);
\node(P3) at (7,1)[inner sep=1 pt]{\small \(P_3\)};
\draw[->](2-3)--(P3);
\draw[->](P3)--(2-4);
\node(P4) at (7,5)[inner sep=1 pt]{\small \(P_4\)};
\draw[->](4-3)--(P4);
\draw[->](P4)--(4-4); 
\replacevertex {(1-0)}{[fill=white, inner sep=2pt]{\(\star\)}}
\replacevertex {(2-5)}{[fill=white, inner sep=2pt]{\(\star\)}}
\replacevertex {(2-0)}{[fill=white, inner sep=2pt]{\(\circ\)}}
\replacevertex {(4-5)}{[fill=white, inner sep=2pt]{\(\circ\)}}
  
\end{tikzpicture}\caption{\(\modf D_{4,1,3}\). The quiver is glued together by identifying the matching symbols on either side.}\label{ARD_{4,1,3}}
\end{figure}

\end{example}

In general the AR-quiver of the stable module category of algebras of type \((\!\D_4,s,3\!)\) is of the form \(\mathbb Z\D_n/\theta^{5s}\chi\), where \(\chi\) is the automorphism of order \(3\) described in Definition \ref{Dtwist}.

\begin{proposition}
Let \(n=4\) and \(s\in\mathbb N\). The categories \(\smod D_{4,s,3}\) and \(\mathcal D^b(k\D_4)/\tau^{5s}\rho\) are equivalent as triangulated categories.
\end{proposition}

\begin{proof}
Since \(5s>0\), the functor \(\tau^{5s}\rho\) satisfies the conditions given in theorem \ref{KellerTriOrb}. Hence the category \(\mathcal D^b(k\!\D_4\!)/\tau^{5s}\rho\) is triangulated. 
The algebra \(D_{n,s,3}\) is a representation-finite, self-injective, basic algebra, whose stable module category is standard by theorem \ref{standardRFS}. The conclusion follows from Corollary \ref{sameARquiver}.
\end{proof}

\subsection{Type \((\D_{3m},\frac{s}{3},1)\) }\label{D3m} 
This is the only type of tree type \(\D\) where the frequency is not an integer. If \(3|s\), then the type is already described, in section \ref{Drot1}; hence we require that \(s\) is not divisible by 3.

\begin{definition}
Let \(m\geq 2\) and \(s\geq1\) with \(3\nmid s\). The representative of self-injective algebras of type \((\D_{3m},\frac{s}{3},1)\) is given by the path algebra \(D_{{3m},\frac{s}{3},1}:=kQ/I\) where \(Q\) is the quiver of figure \ref{koggerDhoved} and the ideal \(I\) is generated by the following set of relations:
\begin{enumerate}
\item \(\alpha^i_m\cdots\alpha^i_2\alpha^i_1=\beta_{i+1}\beta_i\) for all \(i\in\{1,\ldots,s\}=\mathbb Z/\langle s\rangle\);
\item \(\alpha^{i+2}_1\alpha^i_m=0\) for all \(i\in\{1,\ldots,s\}=\mathbb Z/\langle s \rangle;\)
\item \(\alpha^{i+3}_j\!\cdots\alpha^{i+3}_1\beta_{i+2}\alpha^i_m\!\cdots\alpha^i_j\!=\!0\) for all \(i\in\{\!1,\ldots, s\!\}=\mathbb Z/\!\langle s \rangle\) and for all \(j\in\{\!1,\ldots,m\!\}\)
\end{enumerate}
\end{definition}

\begin{figure}[ht]
\centering
\begin{tikzpicture}[font=\footnotesize, scale=1.1]
\node(b1) at (0,0){\(\circ\)};
\node[rotate around={45:(0,-3)}](b2) at (0,0){\(\circ\)};
\node[rotate around={90:(0,-3)}](b3) at (0,0){\(\circ\)};
\node[rotate around={135:(0,-3)}](b4) at (0,0){\(\circ\)};
\node[rotate around={180:(0,-3)}](b5) at (0,0){\(\circ\)};
\node[rotate around={225:(0,-3)}](b6) at (0,0){\(\circ\)};
\node[rotate around={255:(0,-3)}](b7) at (0,0){};
\node[rotate around={265:(0,-3)}](b8) at (0,0){};
\node[rotate around={275:(0,-3)}](b9) at (0,0){};
\node[rotate around={285:(0,-3)}](b10) at (0,0){};
\node[rotate around={315:(0,-3)}](b11) at (0,0){};
\draw[->](b1)--node[below]{\(\beta_1\)}(b2);
\draw[->](b2)--node[right]{\(\beta_2\)}(b3);
\draw[->](b3)--node[right]{\(\beta_3\)}(b4);
\draw[->](b4)--node[above]{\(\beta_4\)}(b5);
\draw[->](b5)--node[above]{\(\beta_5\)}(b6);
\draw[->](b6)--node[left]{\(\beta_6\)}(b7);
\draw[dotted](b7)--(b8);
\draw[dotted](b9)--(b10);
\draw[->](b10)--node[left]{\(\beta_{s-1}\)}(b11);
\draw[->](b11)--node[below]{\(\beta_{s}\)}(b1);
\node(a0) at (0,1){};
\node[rotate around={15:(0,-4)}](a1) at (0,1){\(\circ\)};
\node[rotate around={30:(0,-4)}](a2) at (0,1){\(\circ\)};
\node[rotate around={60:(0,-4)}](a3) at (0,1){\(\circ\)};
\node[rotate around={75:(0,-4)}](a4) at (0,1){\(\circ\)};
\node[rotate around={105:(0,-4)}](a5) at (0,1){\(\circ\)};
\node[rotate around={120:(0,-4)}](a6) at (0,1){\(\circ\)};
\node[rotate around={150:(0,-4)}](a7) at (0,1){\(\circ\)};
\node[rotate around={165:(0,-4)}](a8) at (0,1){\(\circ\)};
\node[rotate around={195:(0,-4)}](a9) at (0,1){\(\circ\)};
\node[rotate around={210:(0,-4)}](a10) at (0,1){\(\circ\)};
\node[rotate around={240:(0,-4)}](a11) at (0,1){};
\node[rotate around={300:(0,-4)}](a12) at (0,1){};
\node[rotate around={330:(0,-4)}](a13) at (0,1){\(\circ\)};
\node[rotate around={345:(0,-4)}](a14) at (0,1){\(\circ\)};
\node(A0) at (0,2){};
\node[rotate around={35:(0,-5)}](A1) at (0,2){};
\node[rotate around={55:(0,-5)}](A2) at (0,2){};
\node[rotate around={80:(0,-5)}](A3) at (0,2){};
\node[rotate around={100:(0,-5)}](A4) at (0,2){};
\node[rotate around={125:(0,-5)}](A5) at (0,2){};
\node[rotate around={145:(0,-5)}](A6) at (0,2){};
\node[rotate around={170:(0,-5)}](A7) at (0,2){};
\node[rotate around={190:(0,-5)}](A8) at (0,2){};
\node[rotate around={215:(0,-5)}](A9) at (0,2){};
\node[rotate around={235:(0,-5)}](A10) at (0,2){};
\node[rotate around={255:(0,-5)}](A11) at (0,2){};
\node[rotate around={265:(0,-5)}](A12) at (0,2){};
\node[rotate around={275:(0,-5)}](A17) at (0,2){};
\node[rotate around={285:(0,-5)}](A18) at (0,2){};
\node[rotate around={305:(0,-5)}](A13) at (0,2){};
\node[rotate around={330:(0,-5)}](A14) at (0,2){};
\node[rotate around={350:(0,-5)}](A15) at (0,2){};
\node[rotate around={15:(0,-5)}](A16) at (0,2){};
\draw[->](b1)--node[left, near start]{\(\alpha^{1}_{1}\)}(a1);
\draw[->](a2)--node[right]{\(\alpha^{s}_{m}\)}(b2);
\draw[->](b2)--node[below]{\(\alpha^{2}_{1}\)}(a3);
\draw[->](a4)--node[right, near start]{\(\alpha^{1}_{m}\)}(b3);
\draw[->](b3)--node[below, near start]{\(\alpha^{3}_{1}\)}(a5);
\draw[->](a6)--node[above]{\(\alpha^{2}_{m}\)}(b4); 
\draw[->](b4)--node[right]{\(\alpha^{4}_{1}\)}(a7);
\draw[->](a8)--node[above, near start]{\(\alpha^{3}_{m}\)}(b5);
\draw[->](b5)--node[above, near end]{\(\alpha^{5}_{1}\)}(a9);
\draw[->](a10)--node[left]{\(\alpha^{4}_{m}\)}(b6);
\draw[->](b6)--node[above]{\(\alpha^{6}_{1}\)}(a11);
\draw[->](a12)--node[below]{\(\alpha^{s-2}_{m}\)}(b11);
\draw[->](b11)--node[left]{\(\alpha^{s}_{1}\)}(a13);
\draw[->](a14)--node[right]{\(\alpha^{s-1}_{m}\)}(b1);

\draw[->](a1)--node[above]{\(\alpha^{1}_{2}\)}(A1);
\draw[dotted](A1)--(A2);
\draw[->](A2)--node[left, near start]{\(\alpha^{1}_{m-1}\)}(a4);
\draw[->](a3)--node[below, near end]{\(\alpha^{2}_{2}\)}(A3);
\draw[dotted](A3)--(A4);
\draw[->](A4)--node[below, near start]{\(\alpha^{2}_{m-1}\)}(a6);
\draw[->](a5)--node[right, very near end]{\(\alpha^{3}_{2}\)}(A5);
\draw[dotted](A5)--(A6);
\draw[->](A6)--node[above, very near start]{\(\alpha^{3}_{m-1}\)}(a8);
\draw[->](a7)--node[above, very near end]{\(\alpha^{4}_{2}\)}(A7);
\draw[dotted](A7)--(A8);
\draw[->](A8)--node[right, very near start]{\(\alpha^{4}_{m-1}\)}(a10);
\draw[->](a9)--node[above, very near end]{\(\alpha^{5}_{2}\)}(A9);
\draw[dotted](A9)--(A10);
\draw[->](a11)--node[above]{\(\alpha^{6}_{2}\)}(A11);
\draw[dotted](A11)--(A12);
\draw[->](A16)--node[right,very near start]{\(\alpha^{s}_{m-1}\)}(a2);
\draw[dotted](A15)--(A16);
\draw[->](a13)--node[left, near end]{\(\alpha^{s}_{2}\)}(A15);
\draw[dotted](A13)--(A14);
\draw[->](A14)--node[below, very near start]{\(\alpha^{s-1}_{m-1}\)}(a14);
\draw[dotted](A17)--(A18);
\draw[->](A18)--node[below]{\(\alpha^{s-2}_{m-1}\)}(a12);

\end{tikzpicture}\caption{\((\D_{3m},\frac{s}{3})\)}\label{koggerDspesial}
\end{figure}

\begin{example}

\begin{figure}[ht]
\centering
\begin{tikzpicture}
\node(1) at (0,0)[inner sep=1 pt]{\(1\)};
\node(2) at (2,0)[inner sep=1 pt]{\(2\)};

\draw[->] (2) .. controls (1.5,0.5) and (0.5,.5)  .. node[above]{\footnotesize \(\alpha_2\)}(1);
\draw[->] (1) .. controls (.5,-0.5) and (1.5,-.5)  .. node[below]{\footnotesize \(\alpha_1\)}(2);
\draw[->] (1) .. controls (-1,-1.25) and (-1,1.25) .. node[left]{\footnotesize \(\beta\)}(1);
\end{tikzpicture}\caption{Quiver of the path algebra \(D_{6,\frac{1}{3},1}\)}\label{D_{6,1/3,1}kogger}
\end{figure}
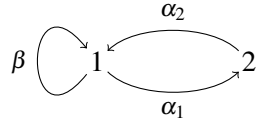

Let \(m=2\) and \(s=1\). The algebra \(D_{6,\frac{1}{3},1}\) is given by the quiver in figure \ref{D_{6,1/3,1}kogger} with relations:
\begin{align*}
\beta^2&=\alpha_2\alpha_1 & \alpha_1\alpha_2&=0\\
\alpha_1\beta\alpha_2\alpha_1&=0 & \alpha_2\alpha_1\beta\alpha_2&=0.
\end{align*}
The AR-quiver of this algebra is shown in figure \ref{ARD_{6,1/3,1}}.  In this case \(\smod D_{6,\frac{1}{3},1}\) is triangle equivalent to \(\mathcal D^b(k\D_6)/\tau^3\).

\begin{figure}[ht]
\centering
 \begin{tikzpicture}[scale=.75,yscale=-1]
 \foreach \x in {0,...,3}
  \foreach \y in {1,2,4,6}
   \node (\y-\x) at (\x*2,\y) [vertex] {};
 \foreach \x in {0,...,2}
  \foreach \y in {3,5}
   \node (\y-\x) at (\x*2+1,\y) [vertex] {};
 \foreach \xa/\xb in {0/1,1/2,2/3}
  \foreach \ya/\yb in {1/3,2/3,4/3,4/5,6/5}
   {
    \draw [->] (\ya-\xa) -- (\yb-\xa);
    \draw [->] (\yb-\xa) -- (\ya-\xb);
   }
 \draw [dashed] (0,.5) -- (0,6.5); 
 \draw [dashed] (6,.5) -- (6,6.5);  
\node (P1) at (3,1)[inner sep=1 pt]{\small \(P_1\)};
\draw[->] (2-1) -- (P1);
\draw[->] (P1) --(2-2);
\node (P2) at (3,7)[inner sep=1 pt]{\small \(P_2\)};
\draw[->] (6-1)--(P2);
\draw[->] (P2)--(6-2);
  
\end{tikzpicture}\caption{\(D_{6,\frac{1}{3},1}\)}\label{ARD_{6,1/3,1}}
\end{figure}
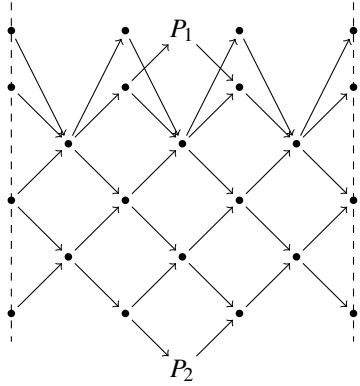
\end{example}

The AR-quiver of the stable module category of algebras of type \((\D_{3m},\frac{s}{3},1)\) is of the form \(\mathbb Z\D_{3m}/\theta^{s(h-1)/3}\), where \(h\) is the Coxeter number for \(\D_{3m}\). (Note that since \(h-1=2n-3=6m-3\) we have that \(s(h-1)/3\) is a natural number).

\begin{proposition}
Let \(m\geq 2\) and \(s\geq1\) with \(3\!\nmid\! s\). Then the categories \(\smod D_{{3m},\frac{s}{3},1}\) and \(\mathcal D^b(k\D_{3m})/\tau^{s(h-1)/3}\) are equivalent as triangulated categories.
\end{proposition}

\begin{proof}
Since \(s(h-1)/3>0\), the functor \(\tau^{s(h-1)/3}\) satisfies the conditions of theorem \ref{KellerTriOrb}, hence the category \(\mathcal D^b(k\D_{3m})/\tau^{s(h-1)/3}\) is triangulated. 
The algebra \(D_{{3m},\frac{s}{3},1}\) is a representation-finite, self-injective, basic algebra, whose stable module category is standard by theorem \ref{standardRFS}. The conclusion follows from Corollary \ref{sameARquiver}.
\end{proof}

\section{Type \(\E\)}\label{typeE}
We now look at self-injective algebras with AR-quivers of the form \(\mathbb Z\E_n/G\). These algebras are all standard \cite{asashiba2}, and they are divided into two main groups; those with a cylindrical AR-quiver, and those with a M\"obius-shaped AR-quiver. In Asashiba's notation, the former are of type \((E_n, s, 1)\), while the latter are of type \((\E_6, s, 2\)), see \cite{asashiba2}.
For the first group, the stable module categories will be equivalent to orbit categories using functors that are some power of the AR-translation \(\tau\). For the latter, however, we need a "flip functor" to get the M\"obius shape of the quiver.

\begin{definition}\label{Eflip}
The flip functor \(\phi\) on \(\mathcal D^b(k\E_6)\) is given by \(\phi=\tau^6 [1]\).
\end{definition}

We follow the classification due to Asashiba for the rest of the section. Note that the representative algebras all share the quiver given in figure \ref{koggerE}; however the relations are different.
\begin{figure}[ht]
\centering
\begin{tikzpicture}[font=\footnotesize, scale=1.1]
\node(g1) at (-1.5,1.5){\(\circ\)}; 
\node(g2) at (-1.5,0){\(\circ\)};
\node(g3) at (-1.5,-1.5){\(\circ\)}; 
\node(g4) at (0,-1.5){\(\circ\)};
\node(g5) at (1.5,-1.5){\(\circ\)}; 
\node(g6) at (1.5,0){};
\node(g7) at (1.5,1){};
\node(g8) at (1,1.5){};
\node(g9) at (0,1.5){};
\draw[->](g1)--node[right]{\(\gamma^{0}_{2}\)}(g2);
\draw[->](g2)--node[right]{\(\gamma^{0}_{1}\)}(g3);
\draw[->](g3)--node[above]{\(\gamma^{1}_{2}\)}(g4);
\draw[->](g4)--node[above]{\(\gamma^{1}_{1}\)}(g5);
\draw[->](g5)--node[left]{\(\gamma^{2}_{2}\)}(g6);
\draw[dotted](g6)--(g7);
\draw[dotted](g8)--(g9);
\draw[->](g9)--node[below]{\(\gamma^{s-1}_{1}\)}(g1);
\node(b1) at (-2.5,0.75){\(\circ\)};
\node(b2) at (-2.5,-0.75){\(\circ\)};
\node(b3) at (-0.75,-2.5){\(\circ\)};
\node(b4) at (0.75,-2.5){\(\circ\)};
\node(b5) at (2.5,-0.75){};
\node(b6) at (2.5, 0.25){};
\node(b7) at (2.5,1.25){};
\node(b8) at (1.25,2.5){};
\node(b9) at (0.25,2.5){};
\node(b10) at (-0.75,2.5){\(\circ\)};
\draw[->](g1)--node[below]{\(\beta^{0}_{3}\)}(b1);
\draw[->](b1)--node[left]{\(\beta^{0}_{2}\)}(b2);
\draw[->](b2)--node[above]{\(\beta^{0}_{1}\)}(g3);
\draw[->](g3)--node[right]{\(\beta^{1}_{3}\)}(b3);
\draw[->](b3)--node[below]{\(\beta^{1}_{2}\)}(b4);
\draw[->](b4)--node[left]{\(\beta^{1}_{1}\)}(g5);
\draw[->](g5)--node[right, near start]{\(\beta^{2}_{3}\)}(b5);
\draw[dotted](b5)--(b6);
\draw[dotted](b7)--(b8);
\draw[dotted](b9)--(b10);
\draw[->](b10)--node[right]{\(\beta^{s-1}_{1}\)}(g1);
\node(a1) at (-3.5,1.5){\(\circ\)};
\node(a2) at (-3.5,0.75){};
\node(a3) at (-3.5,-0.75){};
\node(a4) at (-3.5,-1.5){\(\circ\)};
\node(a5) at (-1.5,-3.5){\(\circ\)};
\node(a6) at (-0.75,-3.5){};
\node(a7) at (0.75,-3.5){};
\node(a8) at (1.5,-3.5){\(\circ\)};
\node(a9) at (3.5, -1.5){\(\circ\)};
\node(a10) at (3.5,-0.75){};
\node(a11) at (3.5, 0.25){};
\node(a12) at (0.25, 3.5){};
\node(a13) at (-0.75,3.5){};
\node(a14) at (-1.5,3.5){\(\circ\)};
\draw[->](g1)--node[above]{\(\alpha^{0}_{n-3}\)}(a1);
\draw[->](a1)--node[left]{\(\alpha^{0}_{n-4}\)}(a2);
\draw[dotted](a2)--(a3);
\draw[->](a3)--node[left]{\(\alpha^{0}_{2}\)}(a4);
\draw[->](a4)--node[below]{\(\alpha^{0}_{1}\)}(g3);
\draw[->](g3)--node[left]{\(\alpha^{1}_{n-3}\)}(a5);
\draw[->](a5)--node[below]{\(\alpha^{1}_{n-4}\)}(a6);
\draw[dotted](a6)--(a7);
\draw[->](a7)--node[below]{\(\alpha^{1}_{2}\)}(a8);
\draw[->](a8)--node[right]{\(\alpha^{1}_{1}\)}(g5);
\draw[->](g5)--node[below]{\(\alpha^{2}_{n-3}\)}(a9);
\draw[->](a9)--node[right]{\(\alpha^{2}_{n-4}\)}(a10);
\draw[dotted](a10)--(a11);
\draw[dotted](a12)--(a13);
\draw[->](a13)--node[above]{\(\alpha^{s-1}_{2}\)}(a14);
\draw[->](a14)--node[left]{\(\alpha^{s-1}_{1}\)}(g1);
\end{tikzpicture}\caption{Type \((\E_n,s)\)}\label{koggerE}
\end{figure}

\subsection{Type \((\E_n,s,1)\)}\label{E678}

\begin{definition}\label{E6s1}
The representative of self-injective algebras of type \((\E_n,s,1)\) is given by the path algebra \(E_{n,s,1}:=kQ/I\) where \(Q\) is the quiver of figure \ref{koggerE} and the ideal \(I\) is generated by the following set of relations:
\begin{enumerate}
\item \(\alpha^i_1\alpha^i_2\cdots\alpha^i_{n-3}=\beta^i_1\beta^i_2\beta^i_3=\gamma^i_1\gamma^i_2\) for all \(i\in\{0,\ldots,s-1\}\);
\item For all \(i\in\{0,\ldots,s-1\}=\mathbb Z/\langle s \rangle\),
\begin{align*}
\beta^{i+1}_3\alpha^i_1&=0,\; &\gamma^{i+1}_2\alpha^i_1&=0,\\
\alpha^{i+1}_{n-3}\beta^i_1&=0,\; &\gamma_2^{i+1}\beta^i_1&=0,\\
\alpha^{i+1}_{n-3}\gamma^i_1&=0,\; &\beta^{i+1}_3\gamma^i_1&=0, \text{ and}
\end{align*}
\item \(\alpha\)-paths of length \(n-2\) are equal to 0, \(\beta\)-paths of length \(4\) are equal to 0 and \(\gamma\)-paths of length \(3\) are equal to 0.
\end{enumerate}
\end{definition}

\begin{example}
\begin{figure}[ht]
\centering
\begin{tikzpicture}
\node(1) at (0,0)[inner sep=1 pt]{\(1\)};
\node(2) at (5,1.5)[inner sep=1 pt]{\(2\)};
\node(3) at (5,-1.5)[inner sep=1 pt]{\(3\)};
\node(4) at (4,.8)[inner sep=1 pt]{\(4\)};
\node(5) at (4,-.8)[inner sep=1 pt]{\(5\)};
\node(6) at (2,0)[inner sep=1 pt]{\(6\)};

\draw[->] (1) .. controls (1.5,1.75) and (4.5,1.75) .. node[above]{\footnotesize \(\alpha_3\)}(2);
\draw[->] (2) .. controls (6.25,0.75) and (6.25,-0.75) ..node[right ]{\footnotesize \(\alpha_2\)}(3);
\draw[->] (3) .. controls (4.5,-1.75) and (1.5,-1.75) .. node[below]{\footnotesize \(\alpha_1\)}(1);

\draw[->] (1) .. controls (1,1.25) and (3,1.25) .. node[below, near end]{\footnotesize\(\beta_3\)}(4);
\draw[->] (4) .. controls (4.8, .4)  and (4.8, -.4) .. node[left]{\footnotesize\(\beta_2\)} (5);
\draw[->] (5) .. controls (3,-1.25) and (1,-1.25) .. node[above, near start]{\footnotesize\(\beta_1\)}(1);

\draw[->] (1) .. controls (0.5,0.5) and (1.5,.5)  .. node[above,very near end]{\footnotesize \(\gamma_2\)}(6);
\draw[->] (6) .. controls (1.5,-0.5) and (.5,-.5)  .. node[below,very  near start]{\footnotesize \(\gamma_1\)}(1);
\end{tikzpicture}\caption{Quiver of \(E_{6,1,n}\) for \(n=1,2\)}\label{ARE_6,1}
\end{figure}
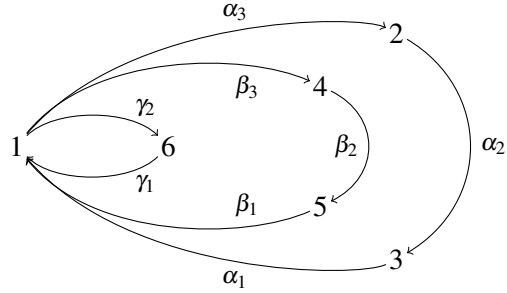
Let \(n=6\) and \(s=1\). The algebra \(E_{6,1,1}\) is given by the quiver in figure \ref{ARE_6,1}, together with the relations 
\[\alpha_1\alpha_2\alpha_3=\beta_1\beta_2\beta_3=\gamma_1\gamma_2\]
\begin{align*}
\alpha_3\beta_1=&0 	&\alpha_3\gamma_1&=0 	&\beta_3\alpha_1&=0\\
\beta_3\gamma_1=&0 	&\gamma_2\alpha_1&=0 	&\gamma_2\beta_1&=0\\
\alpha_2\alpha_3\alpha_1\alpha_2=&0 & \beta_2\beta_3\beta_1\beta_2&=0.
\end{align*}
The AR-quiver of the module category over this algebra is given in figure \ref{ARE611}. It turns out that \(\smod E_{6,1,1}\) is triangulated equivalent to \(\mathcal D^b(k\E_6)/\tau^{11}\).
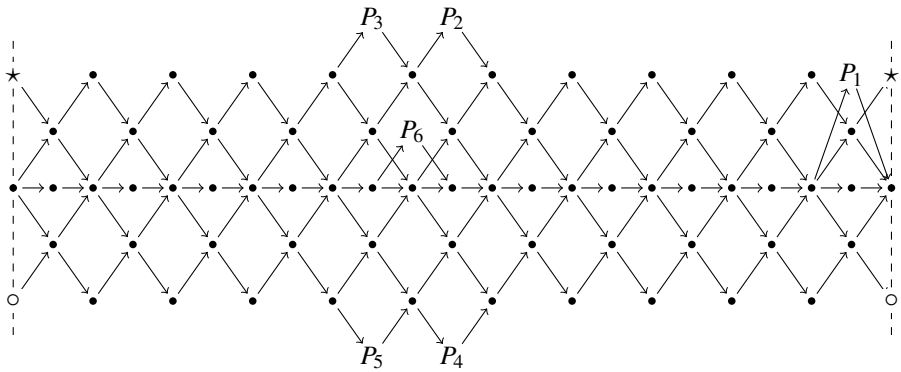
\begin{figure}[ht]
\centering
\begin{tikzpicture}[xscale=.35,yscale=-.75]
 \foreach \x in {0,...,11}
  \foreach \y in {1,3,5}
   \node (\y-\x) at (\x*3,\y) [vertex] {};
 \foreach \x in {0,...,10}
  \foreach \y in {2,4}
   \node (\y-\x) at (\x*3+1.5,\y) [vertex] {};
 \foreach \x in {0,...,10}
  \foreach \y in {3}
   \node (a-\x) at (\x*3+1.5,\y)[vertex] {};
  
 \foreach \xa/\xb in {0/1,1/2,2/3,3/4,4/5,5/6,6/7,7/8,8/9,9/10,10/11}
  \foreach \ya/\yb in {1/2,3/2,3/4,5/4}
   {
    \draw [->] (\ya-\xa) -- (\yb-\xa);
    \draw [->] (\yb-\xa) -- (\ya-\xb);
   }
 \foreach \xa/\xb in {0/1,1/2,2/3,3/4,4/5,5/6,6/7,7/8,8/9,9/10,10/11}
  \foreach \ya/\yb in {3/3}
  {
   \draw [->] (\ya-\xa) -- (a-\xa);
   \draw [->]  (a-\xa) -- (\yb-\xb);
  }

\node (P1) at (31.5, 1)[inner sep=1 pt, font=\small] {\(P_1\)};
\draw[->](3-10) -- (P1);
\draw[->](P1)  -- (3-11);
\node (P2) at (16.5, 0)[inner sep=1 pt, font=\small] {\(P_2\)};
\draw[->](1-5) -- (P2);
\draw[->](P2)  -- (1-6);
\node (P3) at (13.5, 0)[inner sep=1 pt, font=\small] {\(P_3\)};
\draw[->](1-4) -- (P3);
\draw[->](P3)  -- (1-5);
\node (P4) at (16.5, 6)[inner sep=1 pt, font=\small] {\(P_4\)};
\draw[->](5-5) -- (P4);
\draw[->](P4)  -- (5-6);
\node (P5) at (13.5, 6)[inner sep=1 pt, font=\small] {\(P_5\)};
\draw[->](5-4) -- (P5);
\draw[->](P5)  -- (5-5);
\node (P6) at (15, 2)[inner sep=1 pt, font=\small] {\(P_6\)};
\draw[->](a-4) -- (P6);
\draw[->](P6)  -- (a-5);

 \draw [dashed] (0,0.4) -- (0,5.6); 
 \draw [dashed] (33,0.4) -- (33,5.6);
  \replacevertex{(1-0)}  {[fill=white, inner sep=2pt]{\(\star\)}}
  \replacevertex{(1-11)} {[fill=white, inner sep=2pt]{\(\star\)}}
  \replacevertex{(5-0)}  {[fill=white, inner sep=2pt]{\(\circ\)}}
  \replacevertex{(5-11)} {[fill=white, inner sep=2pt]{\(\circ\)}}
\end{tikzpicture}
\caption{AR-quiver of \(\modf E_{6,1,1}\). The quiver is glued together by identifying the matching symbols on either side.}\label{ARE611}
\end{figure}
\end{example}

In general, the AR-quiver of the stable module categories of self-injective algebras of type \((\E_n,s,1)\) is isomorphic to \(\mathbb Z\E_n/\theta^{t_ns}\), where \(t_6=11\), \(t_7=17\) and \(t_8=29\).

\begin{proposition}
Let \(n=6,7,8\) and \(s\geq 1\). The categories \(\smod E_{n,s,1}\) and \(\mathcal D^b(k\E_n)/\tau^{t_ns}\) are triangle equivalent.
\end{proposition}
\begin{proof}
Since \(t_ns>0\), the functor \(\tau^{t_ns}\) satisfies the conditions of theorem \ref{KellerTriOrb}, hence \(\mathcal D^b(k\E_n)/\tau^{t_ns}\) is a triangulated category.
The algebra \(E_{n,s,1}\) is a representation-finite, self-injective, basic algebra, whose stable module category is standard by theorem \ref{standardRFS}. The conclusion follows from Corollary \ref{sameARquiver}.
\end{proof}

\subsection{Type \((\E_6,s,2)\)}\label{E6twist}
\begin{definition}
The representative of self-injective algebras of type \((\E_6,s,2)\) is given by the path algebra \(E_{6,s,2}:=kQ/I\) where \(Q\) is the quiver of figure \ref{koggerE} and the ideal \(I\) is generated by the following set of relations:
\begin{enumerate}

\item \(\alpha^i_1\alpha^i_2\cdots\alpha^i_{n-3}=\beta^i_1\beta^i_2\beta^i_3=\gamma^i_1\gamma^i_2\) for all \(i\in\{0,\ldots,s-1\}\);
\item For all \(i\in\{0,\ldots,s-1\}=\mathbb Z/\langle s\rangle,\)
\begin{align*}
\alpha^{i+1}_3\gamma^i_1&=0,\; & \beta^{i+1}_3\gamma^i_1&=0,\\
\gamma^{i+1}_2\alpha^i_1&=0, \; & \gamma^{i+1}_2\beta^i_1&=0,
\end{align*}
and for all \(i\in\{0,\ldots,s-2\}\),
\begin{align*}
\beta^{i+1}_3\alpha^i_1=&0,\; &\alpha^{i+1}_3\beta^i_1&=0,\\
\alpha^0_3\alpha^{s-1}_1=&0,\; & \beta^0_3\beta^{s-1}_1&=0, \text{ and}
\end{align*}
\item\(\gamma\)-paths of length 3 are equal to 0, and for all \(i\in\{0,\ldots,s-2\}\) and for all \(j\in\{1,2,3\}=\mathbb Z/\langle 3\rangle\),
\begin{align*}
 \alpha^{i+1}_{j-3}\cdots\alpha^i_j&=0,\; &\beta^{i+1}_{j-3}\cdots\beta^i_j&=0,\\
 \beta^0_{j-3}\cdots\beta^0_3\alpha^{s-1}_1\cdots\alpha^{s-1}_j&=0,\; &\alpha^0_{j-3}\cdots\alpha^0_3\beta^{s-1}_1\cdots\beta^{s-1}_j&=0.
\end{align*}
\end{enumerate}
\end{definition}

\begin{example}
Let \(n=6\) and \(s=1\). The algebra \(E_{6,1,2}\) is given by the quiver in figure \ref{ARE_6,1}, together with the relations 
\[\alpha_1\alpha_2\alpha_3=\beta_1\beta_2\beta_3=\gamma_1\gamma_2\]
\begin{align*}
\alpha_3\alpha_1=&0 	&\alpha_3\gamma_1&=0 	&\beta_3\beta_1&=0\\
\beta_3\gamma_1=&0 	&\gamma_2\alpha_1&=0 	&\gamma_2\beta_1&=0\\
\alpha_2\alpha_3\beta_1\beta_2=&0 & \beta_2\beta_3\alpha_1\alpha_2&=0 
\end{align*}
The AR-quiver of the module category over this algebra is given in figure \ref{ARE612}. It turns out that \(\smod E_{6,1,2}\) is triangulated equivalent to \(\mathcal D^b(k\E_6)/\tau^{11}\phi\).
\begin{figure}[ht]
\centering
\begin{tikzpicture}[xscale=.35,yscale=-.75]
 \foreach \x in {0,...,11}
  \foreach \y in {1,3,5}
   \node (\y-\x) at (\x*3,\y) [vertex] {};
 \foreach \x in {0,...,10}
  \foreach \y in {2,4}
   \node (\y-\x) at (\x*3+1.5,\y) [vertex] {};
 \foreach \x in {0,...,10}
  \foreach \y in {3}
   \node (a-\x) at (\x*3+1.5,\y)[vertex] {};
  
 \foreach \xa/\xb in {0/1,1/2,2/3,3/4,4/5,5/6,6/7,7/8,8/9,9/10,10/11}
  \foreach \ya/\yb in {1/2,3/2,3/4,5/4}
   {
    \draw [->] (\ya-\xa) -- (\yb-\xa);
    \draw [->] (\yb-\xa) -- (\ya-\xb);
   }
 \foreach \xa/\xb in {0/1,1/2,2/3,3/4,4/5,5/6,6/7,7/8,8/9,9/10,10/11}
  \foreach \ya/\yb in {3/3}
  {
   \draw [->] (\ya-\xa) -- (a-\xa);
   \draw [->]  (a-\xa) -- (\yb-\xb);
  }

\node (P1) at (31.5, 1)[inner sep=1 pt, font=\small] {\(P_1\)};
\draw[->](3-10) -- (P1);
\draw[->](P1)  -- (3-11);
\node (P2) at (16.5, 0)[inner sep=1 pt, font=\small] {\(P_2\)};
\draw[->](1-5) -- (P2);
\draw[->](P2)  -- (1-6);
\node (P3) at (13.5, 0)[inner sep=1 pt, font=\small] {\(P_3\)};
\draw[->](1-4) -- (P3);
\draw[->](P3)  -- (1-5);
\node (P4) at (16.5, 6)[inner sep=1 pt, font=\small] {\(P_4\)};
\draw[->](5-5) -- (P4);
\draw[->](P4)  -- (5-6);
\node (P5) at (13.5, 6)[inner sep=1 pt, font=\small] {\(P_5\)};
\draw[->](5-4) -- (P5);
\draw[->](P5)  -- (5-5);
\node (P6) at (15, 2)[inner sep=1 pt, font=\small] {\(P_6\)};
\draw[->](a-4) -- (P6);
\draw[->](P6)  -- (a-5);

 \draw [dashed] (0,0.4) -- (0,5.6); 
 \draw [dashed] (33,0.4) -- (33,5.6);
  \replacevertex{(1-0)}{[fill=white, inner sep=2pt]{\(\star\)}}
  \replacevertex{(5-11)}{[fill=white, inner sep=2pt]{\(\star\)}}
  \replacevertex{(5-0)}{[fill=white, inner sep=2pt]{\(\circ\)}}
  \replacevertex{(1-11)}{[fill=white, inner sep=2pt]{\(\circ\)}}
\end{tikzpicture}
\caption{AR-quiver of \(\modf E_{6,1,2}\). The quiver is glued together by identifying the matching symbols on either side.}
\label{ARE612}
\end{figure}
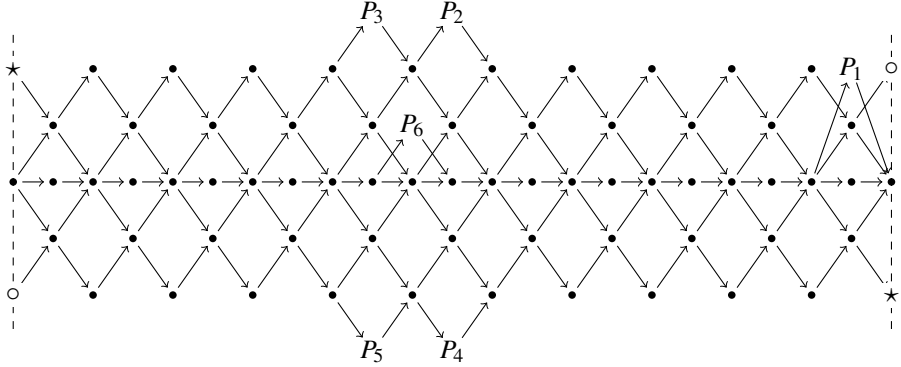
\end{example}

In general, the AR-quiver of the stable module categories of self-injective algebras of type \((\E_6,s,2)\) is isomorphic to \(\mathbb Z\E_6/\theta^{11s}\phi\), where \(\phi\) is described in Table \ref{tableOfS}.

\begin{proposition}
Let \(s\geq 1\). The categories \(\smod E_{6,s,2}\) and \(\mathcal D^b(k\E_6)/\tau^{11s}\phi\) are triangle equivalent.
\end{proposition}
\begin{proof}
Since \(11s>0\), the functor \(\tau^{11s}\phi\) satisfies the conditions of theorem \ref{KellerTriOrb}, hence \(\mathcal D^b(k\E_6)/\tau^{11s}\phi\) is a triangulated category.
The algebra \(E_{6,s,2}\) is a representation-finite, self-injective, basic algebra, whose stable module category is standard by theorem \ref{standardRFS}. The conclusion follows from Corollary \ref{sameARquiver}.
\end{proof}

\section{Summary}

From the propositions of Sections 6, 7 and 8 it is clear that all self-injective standard algebras of finite representation type are stably triangle equivalent to orbit categories of the form \(\mathcal D^b(k\Delta_r)/\tau^w\phi^i\) where \(i\in\{0,1\}\) and \(\phi\) is the functor described in definition \ref{Aflip} for type \(\A\), definition \ref{Dflip} and \ref{Dtwist} for \(\D\) and definition \ref{Eflip} for type \(\E_6\). However not all triangulated orbit categories of the form \(\mathcal D^b(k\Delta_r)/F \) are equivalent to a stable module category of a representation finite self-injective algebra. We therefore sum up our findings in a table below, aiming at a way to easily look up if a certain orbit category is in fact equivalent or not to a stable module category of a self-injective algebra. 

Recall that given a functor of the form \(F=\tau^m[n]\) on \(\mathcal D^b(k\Delta_r)\), it can be expressed on the form \(F=\tau^w\phi^i\) using the Coxeter relation for \(\Delta_r\), and the above-mentioned definitions of \(\phi\). Thus the same autoequivalence, and hence orbit category, can be expressed in many different ways.

The following theorem sums up our results.
\begin{theorem}
Let \(\Delta\) be a Dynkin diagram and let \(\Phi\) be an autoequivalence such that \(\Db(k\Delta)/\Phi\) is triangulated. Let \(\Lambda\) a self-injective algebra.
The orbit category \(\mathcal C=\Db(k\Delta)/\Phi\) is triangle equivalent to \(\smod \Lambda\) exactly in the cases described in table  \ref{SummaryTable}.
\begin{table}[ht]
\begin{center}
\caption{The cases, up to triangulated equivalence, where \(\mathcal C=\Db(k\Delta)/\Phi\) is triangle equivalent to \(\smod \Lambda\).}
\label{SummaryTable}
\begin{tabular}{|>{$}l<{$}>{$}l<{$}|l|l|}
\hline
\multicolumn{2}{|l|}{\(\mathcal{C}\)} 		& \(\Lambda\) & Sec.\!\\\hline
\Db(k\A_r)/\tau^w			&r\geq 1, w\geq 1 												& Nakayama alg.\ \(N_{w,r+1}\)\! & \ref{Nakayama}\\\hline
\Db(k\A_r)/\tau^w\phi	 	&\begin{aligned}r&=2l+1,\, l\geq 1 \\ w&=rv,\, r\geq 1\end{aligned} 	& M\"obius alg.\ \(M_{l,v}\) 	&\ref{Mobius} \\\hline
\Db(k\D_r)/\tau^w	 		&r\geq 4, w=s(2r-3), s\geq 1 				& \(D_{n,s,1}\) 		& \ref{Drot1} \\\hline
\Db(k\D_r)/\tau^w\phi	 	&r\geq 4, w=s(2r-3), s\geq 1				& \(D_{n,s,2}\)		& \ref{Drot2} \\\hline
\Db(k\D_4)/\tau^{5w}\rho 	&w\geq 1									& \(D_{4,s,3}\)		& \ref{Drot3} \\\hline
\Db(k\D_r)/\tau^w	 		&\begin{aligned}r&=3m, m\geq 2\\ w&=s(2r-3)/3, s\geq 1, 3\nmid s\end{aligned}  \!	& \(D_{3m,\frac{s}{3},1}\) & \ref{D3m}\\\hline
\Db(k\E_r)/\tau^w	 		&\begin{aligned}	
						r=6 \text{ and }& w=11s \\ 
						r=7 \text{ and }& w=17s  \hspace{5pt}s\geq 1 \\
						r=8 \text{ and }& w=29s \\ \end{aligned} 	 & \(E_{r, s, 1}\) & \ref{E678}\\\hline
\Db(k\E_6)/\tau^w\phi 		& w=11s, s\geq 1 & \(E_{6,s,2}\) 	& \ref{E6twist} \\\hline
\end{tabular}
\end{center}
\end{table}
\end{theorem}
\subsection*{Acknowledgements}
The authors would like to thank Steffen Oppermann and Aslak Bakke Buan for helpful comments during the initial writing, and the anonymous referee for very useful feedback on the paper. 

\bibliographystyle{plain}
\end{document}